\documentclass[11pt, a4paper]{article} 

\usepackage{geometry}     
\usepackage{amsmath, amsthm, amssymb, amsfonts, bm}
\usepackage{subfig, subfloat, graphicx, tikz}
\graphicspath{{./pic/}}

\newtheorem{example}{Example}

\newtheorem{remark}{Remark}[section]
\newtheorem{proposition}{Proposition}[section]
\newtheorem{lemma}{Lemma}[section]
\newtheorem{theorem}{Theorem}[section]
\newtheorem{definition}{Definition}[section]

\usepackage{booktabs, multirow}
\usepackage{epstopdf}
\usepackage{algorithmic}
\usepackage{cleveref}

\title{Direct and Inverse scattering in a three-dimensional planar waveguide}

\author{Yan Chang\thanks{School of Mathematics, Harbin Institute of Technology,  Harbin 150000, China,
  ({21B312002@stu.hit.edu.cn}).}
\and Yukun Guo\thanks{School of Mathematics, Harbin Institute of Technology,  Harbin 150000, China,
  ({ykguo@hit.edu.cn}).}
\and Yue Zhao\thanks{School of Mathematics and Statistics, Central China Normal University,
	Wuhan 430079, China, ({zhaoyueccnu@163.com}).}}
	
\date{}

\begin{document}

\maketitle

\begin{abstract}
	In this paper, we study the direct and inverse scattering of the Schr\"odinger equation in a three-dimensional planar waveguide. For the direct problem, we derive a resonance-free region and resolvent estimates for the resolvent of the Schr\"odinger operator in such a geometry. Based on the analysis of the resolvent, several inverse problems are investigated. First, given the potential function, we prove the uniqueness of the inverse source problem with multi-frequency data. We also develop a Fourier-based method to reconstruct the source function. The capability of this method is numerically illustrated by examples. Second, the uniqueness and increased stability of an inverse potential problem from data generated by incident waves are achieved in the absence of the source function. To derive the stability estimate, we use an argument of quantitative analytic continuation in complex theory. Third, we prove the uniqueness of simultaneously determining the source and potential by active boundary data generated by incident waves.  In these inverse problems, we only use the limited lateral Dirichlet boundary data at multiple wavenumbers within a finite interval. 
\end{abstract}


\section{Introduction}

We consider the direct and inverse scattering problems in a three-dimensional planar waveguide. Let $D\subset\mathbb R^3$ be an infinite slab between two parallel hyperplanes $\Gamma^+$ and $\Gamma^-$ with width $L$. Without loss of generality, we assume that
\[
D = \{x = (\tilde{x}, x_3)\in\mathbb R^3: \tilde{x}= (x_1, x_2)\in\mathbb R^2,\ 0<x_3<L\}, \quad L>0,
\]
and
\[
\Gamma^+ = \{x\in\mathbb R^3: x_3 = L\}, \quad \Gamma^- = \{x\in\mathbb R^3: x_3 = 0\}.
\]

Consider the following Schr\"odinger equation
\begin{equation}\label{main_eq}
     -\Delta u+ V u -  k^2 u = f,\quad \text{in}\ D,
\end{equation}
where $k>0$ is the wavenumber, $V$ is the potential function, and $f$ is the source function. 
Let $\widetilde{B}_R = \{\tilde{x}\in\mathbb R^2: |\tilde{x}|\leq R\}$ be a two-dimensional disk and $C_R = \widetilde{B}_R\times [0, L]$ be a cylinder in the waveguide. Denote the lateral boundary of $C_R$ by $\Gamma_R = \partial \widetilde{B}_R \times [0, L]$.
Assume that both $f\in L^2(D)$ and $V\in L^\infty(D)$ are compactly supported in $C_R$.
We also assume that the potential function $V$ is invariant in the $x_3$-variable, i.e., $V = V(\tilde{x})$.
Let $u$ satisfy the Neumann and Dirichlet boundary
conditions, respectively,
\begin{equation}\label{bc}
	\frac{\partial u}{\partial x_3} = 0 \quad \text{on} \,\, \Gamma^+, \quad u = 0 \quad \text{on} \,\, \Gamma^-.
\end{equation}

An important application of this waveguide problem is to provide a simplified but effective model for the propagation of time-harmonic acoustic waves in the ocean \cite{AK, Xu}.  
In this model, the Dirichlet boundary condition models the surface of the ocean, while the Neumann boundary condition models the seabed underneath. In this paper, we begin by analyzing the direct scattering problem to seek a resonance-free region and derive the resolvent estimates
for the resolvent $(-\Delta - \lambda^2 + V)^{-1}$ of the Schr\"odinger operator in the waveguide. These studies enable us to gain some insights into the inverse problems. 
As an application of the resolvent estimates, we continue to study the inverse problems of determining the source and potential functions from the knowledge of the scattered field measured on $\Gamma_R$ corresponding to the wavenumber given in a finite interval. The framework developed in this paper is unified and can be applied to the configuration of a tubular waveguide.

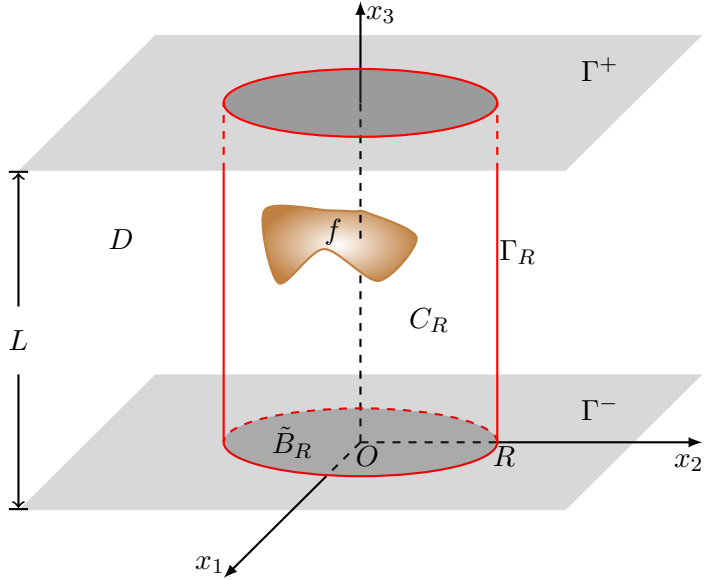
\begin{figure}
\centering
\begin{tikzpicture}[thick, scale=0.9]

\fill [lightgray, opacity = 0.6, xslant = 1] (-1, -4) rectangle (7, -2); 
\fill [lightgray, opacity = 0.6, yshift = 5cm, xslant = 1] (-1, -4) rectangle (7, -2); 

\pgfmathsetseed{66}			
\draw[brown] plot[brown, smooth cycle, samples=8, domain={1:8}, xshift=-.3cm] (\x*360/7+2*rnd:0.1cm+1.6cm*rnd) [shading=ball, outer color=brown, inner color=white]; 
\draw node at (-.4, 0.1) {$f$};

\draw [dashed] (0, -3) -- (2, -3); \draw[-latex] (2, -3) -- (5, -3); 
\draw [dashed] (0, -3) -- (-.5, -3.5); \draw[-latex] (-.5, -3.5) -- (-2, -5);  
\draw [dashed] (0, -3) -- (0, -.5); \draw [dashed] (0, 0) -- (0, 2);\draw[-latex] (0, 2) -- (0, 3.5);  
\draw [->|] (-5, -1) -- (-5, 1);  \draw [->|] (-5, -2) -- (-5, -4);  

\draw [red, dashed] (2, -3) arc (0:180:2 and 0.5); 
\draw [red] (2, -3) arc (0:-180:2 and 0.5); 
\fill [darkgray, opacity = 0.3] (0, -3) ellipse (2 and 0.5); 
\fill [darkgray, opacity = 0.4] (0, 2) ellipse (2 and 0.5);  
\draw [red] (0, 2) ellipse (2 and 0.5); 

\draw [red] (-2, -3)--(-2, 1); \draw [dashed, red] (-2, 1) -- (-2, 2); 
\draw [red] (2, -3)--(2, 1); \draw [dashed, red] (2, 1) -- (2, 2);

\draw (-2.2, -4.8) node {$x_1$};
\draw (4.8, -3.3) node {$x_2$};
\draw (0.3, 3.3) node {$x_3$};
\draw (-1, -3) node {$\tilde{B}_R$};
\draw (-3.5, 0) node {$D$};
\draw (1, -1.2) node {$C_R$};
\draw (0.1, -3.2) node {$O$};
\draw (-5, -1.5) node {$L$};
\draw (2.1, -3.2) node {$R$};
\draw (2.3, -.2) node {$\Gamma_R$};
\draw (3.5, 2.5) node {$\Gamma^+$};
\draw (3.5, -2.5) node {$\Gamma^-$};

\end{tikzpicture} 
\caption{An illustration of the scattering problem in the waveguide.}
\end{figure}

The direct scattering problems in a planar waveguide corresponding to a positive wavenumber have been well studied in the literature. In this case, resonances may occur at a sequence of special frequencies, and this fact leads to the non-uniqueness of the direct problem \cite{AGL}. However, the analysis of the scattering resonances for the resolvent $(-\Delta - \lambda^2 + V)^{-1}$ in this geometry remains open. In scattering theory, the resonances are considered to be the poles of the meromorphic extension of the resolvent with respect to $\lambda$ which could be a complex number. The phenomenon of scattering resonances naturally occurs and has significant applications in a wide range of scientific and engineering areas. For instance, the properties of scattering resonances can be applied to long-time asymptotics of the wave equation, which leads to resonance expansions of waves \cite{Dyatlov}. In this paper, we analyze the direct scattering and derive a resonance-free region and resolvent estimates for the resolvent of the Schr\"odinger operator in a three-dimensional planar waveguide. 

The existing results on the inverse source problem in a waveguide are mainly
focused on the delta-type/point-like sources and numerics \cite{BIT, BGT, L, LXZ}. In \cite{Meng}, the author studied the inverse problem of qualitatively recovering the support of a source function by the multi-frequency far-field pattern with certain assumptions on the source function. To our knowledge, this article makes the novel attempt to establish the theoretical foundation and provide a feasible numerical scheme for quantitatively determining a general source function in the waveguide. We refer to \cite{KLU, LU} and the references cited therein for the corresponding inverse boundary value problems, where the data is the Dirichlet-to-Neumann (DtN) map and the main mathematical tool is the construction of complex geometric optics (CGO) solutions. In this paper, since multi-wavenumber data is available, our treatment does not resort to the construction of CGO solutions.

The first part of this paper is concerned with the direct problem. In our setting, the direct problem is to investigate the resolvent
$(-\Delta - \lambda^2 + V)^{-1}$ where $\lambda$ could be complex. Unfortunately, as discussed above, the scattering waveguides have a special feature that is not present in free space: there may exist a sequence of scattering resonances for which the uniqueness of solutions does not hold. Another special feature is that the scattered field consists of an infinite number of exponentially decaying evanescent modes. This will bring difficulty in the analytic continuation of the resolvent to the lower-half complex plane. To resolve this issue, we assume that the source function only admits a finite number of Fourier modes with respect to the $x_3$ variable. As a consequence, the scattered field also consists of a finite number of Fourier modes. We derive the desired resonance-free region and resolvent estimates by studying the analytic continuation of the complex wavenumber corresponding to each mode and applying the recent results on the resolvent of the Schr\"odinger operator in the free two-dimensional space in \cite{ZZ}. As a special case, we also consider the free resolvent $(-\Delta - \lambda^2)^{-1}$. In this case, we show that a larger sectorial analytic domain can be obtained in the absence of the potential function. 

The study of direct scattering problems paves the way for delving into the inverse problems. The second part of this paper is devoted to three prototypical inverse scattering problems: recovering respectively the source $f$, the potential $V$, and the co-recovery of both $f$ and $V$. The aforementioned analysis of the resolvent is indispensable in establishing the stability of the inverse problems. For the inverse scattering problems by using multi-wavenumber data, a key question shall be usually answered: whether one can find a resonance-free region that contains an infinite interval of the positive real axis. If this is true, then one can apply the analytic continuation principle. Based upon the analysis of the resolvent, the stability and uniqueness issues of the three inverse problems are mathematically investigated. The analysis employs the limited lateral near-field Dirichlet data on $\Gamma_R$ at multiple frequencies. 

For the first inverse problem of determining the source with known potential, we establish the theoretical results of uniqueness and stability. For such problems, the use of multi-wavenumber is usually necessary to overcome the non-uniqueness issue \cite{Bao}. We construct an orthogonal basis in $L^2(C_R)$ and deduce integral equations that connect the scattering data and the source function. As a consequence, we derive uniqueness for the inverse source problem. Similar techniques can be found in, e.g.,\cite{BLZ}. In particular, in the case $V = 0$, we develop a numerical scheme to reconstruct the source from the multi-frequency radiated field. The proposed approach can be viewed as a novel extension of the Fourier method for solving the multi-frequency inverse source problem of acoustic wave \cite{ZG}. The Fourier method has been applied to the inverse source problems for electromagnetic wave \cite{WMGL18, WSGLL19}, elastic wave \cite{SW19, WZSW22}, and recently the biharmonic wave \cite{CGYZ}. Nevertheless, it is not trivial to design the Fourier method for recovering the source in the waveguide because different Fourier modes inherently intertwine with each other. To tackle this obstruction, we adopted the variable separation strategy to first represent the radiated field via a series expression, with the expansion coefficients corresponding to the Fourier modes. Then the Fourier method can be adopted to reconstruct the Fourier modes of the source function, and thus the series expansions can be used as the reconstruction. In addition, it deserves noting that, compared with the direct problem where we assume that the source function only admits a finite number of Fourier modes for the $x_3$ variable, the Fourier method developed here is feasible for recovering a more general source. Furthermore, we also discuss the extension of the inverse source problem with far-field data.
	
Next, we consider the inverse potential problem in the absence of the source function, we prove the uniqueness and increasing stability of the inverse potential problem by the scattered field generated by incident waves. The proof relies on the resolvent estimate and does not resort to the method of constructing CGO solutions as in \cite{KLU, LU} since multi-wavenumber data is available. Based on the resonance-free region and resolvent estimates, the increased stability is obtained by applying an argument of analytic continuation developed in \cite{ZZ}. 

The last inverse problem is the more challenging co-inversion from the active boundary measurements generated by both the endogenous source and the exogenous excitation waves. We show that the unknown potential and source can be uniquely determined simultaneously.  For the simultaneous determination of the source and potential of the random Schr\"odinger equation in free space, we refer the reader to \cite{LLM1, LLM}. We also refer to \cite{BLT21} for the simultaneous recovery of a potential and a point source in the Schr\"odinger equation from Cauchy data. Note that a promising feature of the current study is that only limited-aperture Dirichlet boundary measurements at multiple wavenumbers are needed throughout the analysis. 

The rest of the paper is organized as follows. \Cref{sec: direct_problem} is concerned with the direct scattering problem. An analytic domain and resolvent estimates of the resolvent in the waveguide are derived.  \Crefrange{sec: IP1}{sec: IP3} are dedicated to the inverse problems with the help of the analysis of the resolvent. First, we consider theoretical uniqueness and develop a numerical method for the inverse source problem in \cref{sec: IP1}. Then, in the absence of the source function, the inverse potential scattering problem is investigated in \cref{sec: IP2}. The third inverse problem of identifying the source-potential pair is tackled in \cref{sec: IP2}. Here we present a uniqueness result of simultaneously recovering the source and potential by active measurements. Finally, we conclude this article by summarizing the contributions and shedding light on future works in \cref{sec: conclusion}.

\section{Direct scattering}\label{sec: direct_problem}

In this section, we investigate the resolvent in the planar waveguide. Consider the following Schr\"odinger equation 
\begin{equation}\label{model}
	-\Delta u+ V u- \lambda^2 u =f,\quad \lambda\in\mathbb R^+.
\end{equation}
Denote the resolvent by
\[
R_V(\lambda) = (-\Delta - \lambda^2 + V)^{-1}
\]
which gives
\[
u(x, \lambda) = R_V(\lambda) f(x).
\]

We carry out the separation of variables in $\tilde{x}$ and $x_3$ for the scattered field and the source functions, which leads to series expansions as follows
\begin{equation}
	\label{eq:seriesExpansion}
	u(\tilde{x}, x_3) = \sum_{n = 1}^\infty u_n(\tilde{x}) \sin(\alpha_n x_3),\quad 
	f(\tilde{x}, x_3) = \sum_{n = 1}^\infty f_n(\tilde{x}) \sin(\alpha_n x_3),
\end{equation}
with $\alpha_n = \frac{(2n - 1)\pi}{2L}$ and  the modes $u_n$ of $u$ and $f_n$ of $f$ are given by
\begin{align}\label{fourier}
	u_n(\tilde{x}) = \frac{2}{L} \int_0^L u(x) \sin(\alpha_n x_3) {\rm d}x_3,\quad 
	f_n(\tilde{x}) = \frac{2}{L} \int_0^L f(x) \sin(\alpha_n x_3) {\rm d}x_3.
\end{align}

The modes $u_n$ are required to satisfy the two-dimensional Helmholtz equation 
\begin{equation}\label{eqn_n}
	-\Delta_{\tilde{x}} u_n - \beta_n^2(\lambda) u_n + V u_n = f_n
\end{equation}
where
\[
\beta_n^2(\lambda) = \lambda^2 - \alpha_n^2  \,\, \text{with} \,\, \Im\beta_n(\lambda)\geq 0,
\]
and the Sommerfeld radiation condition proposed in \cite{GX}
\begin{align}\label{src}
	\lim_{r\to\infty} r^{1/2} (\partial_r u_n - {\rm i}\beta_n u_n) = 0, \quad r = |\tilde{x}|.
\end{align}
Denote the resolvent of the Schr\"odinger operator $-\Delta_{\tilde{x}} - \lambda^2 + V$ in two dimensions by
\[
\widetilde{R}_V(\lambda) = (-\Delta_{\tilde{x}} - \lambda^2 + V)^{-1}.
\]
Thus, we have a formal representation of the resolvent 
\begin{align}\label{sum1}
	R_V(\lambda)(f) = \sum_{n = 1}^\infty \widetilde{R}_V(\beta_n(\lambda)) (f_n) \sin(\alpha_n x_3), \quad \lambda\in{\mathbb R}^+\backslash\cup_{n=0}^\infty\{\alpha_n\}. 
\end{align}

In what follows, we shall delineate the analyticity of $\widetilde{R}_V(\beta_n(\lambda))$ and its resolvent estimates concerning complex $\lambda$.

\subsection{Analytic continuation}

In this subsection, we would resort to the analytic continuation technique to extend the domain of $\beta_n(\lambda)$ from $\lambda\in\mathbb R^+$ to the complex plane such that $\beta_n(\lambda)$ is complex analytic.
Let $\lambda = \lambda_1 + {\rm i}\lambda_2\in\mathbb C$. We have
\[
\beta_n^2(\lambda_1, \lambda_2)   = \gamma_n  + {\rm i} \eta
\]
where
\begin{equation}\label{ge}
	\gamma_n  = \lambda_1^2 - \lambda_2^2 - \alpha_n^2, \quad \eta = 2\lambda_1\lambda_2.
\end{equation}
Moreover, a direct calculation yields 
\[
\beta_n(\lambda_1, \lambda_2)  = a_n (\lambda_1, \lambda_2) +{\rm i} b_n (\lambda_1, \lambda_2), \quad \text{for} \,\, \lambda_1\geq 0, \,\, \lambda_2\geq 0,
\]
where
\begin{align*}
	a_n (\lambda_1, \lambda_2) = \left(\frac{\sqrt{\gamma_n^2 + \eta^2} + \gamma_n}{2}\right)^{1/2},\quad 
	b_n (\lambda_1, \lambda_2) = \left(\frac{\sqrt{\gamma_n^2 + \eta^2} - \gamma_n}{2}\right)^{1/2}.
\end{align*}

Now we extend the real analytic functions $a_n$ and $b_n$ analytically from the first quadrant $\{\lambda_1\geq 0, \, \lambda_2\geq 0\}$ to $\mathbb R^2\setminus[-\alpha_{n}, \alpha_{n}]$ by excluding the resonance $\alpha_n$, which will lead to the complex analyticity of $\beta_n(\lambda)$. First, noting that $\partial_{\lambda_2} a_n(\lambda_1, \lambda_2) = 0$ and $b_n(\lambda_1, \lambda_2)= 0$ on $\{\lambda\in\mathbb R^2: \lambda_2 = 0\}\setminus[-\alpha_{n}, \alpha_{n}]$, we apply the even extension to $a_n(\lambda_1, \lambda_2)$ and the odd extension to $b_n(\lambda_1, \lambda_2)$ with respect to the variable $\lambda_2$ which gives the analytic extensions of $a_n$ and $b_n$ from the first quadrant to the right half plane 
$\{\lambda_1\geq 0\} \setminus[0, \alpha_{n}]$ as follows 
\begin{align*}
	\Re\beta_n(\lambda_1, \lambda_2) = 
	\begin{cases}
		a_n(\lambda_1, \lambda_2), \quad &\lambda_2\geq 0,\\[2pt]
		a_n(\lambda_1, -\lambda_2), \quad &\lambda_2< 0,
	\end{cases}
\end{align*}
and
\begin{align*}
	\Im\beta_n (\lambda_1, \lambda_2)= 
	\begin{cases}
		b_n(\lambda_1, \lambda_2), \quad &\lambda_2\geq 0,\\[2pt]
		-b_n(\lambda_1, -\lambda_2), \quad &\lambda_2< 0.
	\end{cases}
\end{align*}

Next, noting that $a_n(\lambda_1, \lambda_2) = 0$ and $\partial_{\lambda_1} b_n(\lambda_1, \lambda_2)= 0$ on the axis $\{\lambda\in\mathbb R^2: \lambda_1 = 0\}$,
we apply the odd extension to $a_n(\lambda_1, \lambda_2)$ and the even extension to $b_n(\lambda_1, \lambda_2)$ with respect to the variable $\lambda_1$ by 
\begin{align*}
	\Re\beta_n(\lambda_1, \lambda_2) = 
	\begin{cases}
		a_n(\lambda_1, \lambda_2), \quad &\lambda_1\geq 0,\\[2pt]
		-a_n(-\lambda_1, \lambda_2), \quad &\lambda_1< 0,
	\end{cases}
\end{align*}
and
\begin{align*}
	\Im\beta_n (\lambda_1, \lambda_2)= 
	\begin{cases}
		b_n(\lambda_1, \lambda_2), \quad &\lambda_1\geq 0,\\[2pt]
		b_n(-\lambda_1, \lambda_2), \quad &\lambda_1< 0.
	\end{cases}
\end{align*}

In this way, we analytically extend $a_n$ and $b_n$ to the left half plane $\{\lambda_1\leq 0\} \setminus[-\alpha_{n}, 0]$.
Therefore, the real part $\Re\beta_n(\lambda_1, \lambda_2)$ and imaginary part $\Im\beta_n (\lambda_1, \lambda_2)$ are both extended
as real analytic functions for $\lambda\in \mathbb R^2\setminus [-\alpha_{n}, \alpha_{n}]$.
As a consequence of the above extension, we can show that $\beta_n (\lambda) = \Re\beta_n (\lambda_1, \lambda_2) + {\rm i} \Im\beta_n (\lambda_1, \lambda_2)$ is complex analytic for 
$\lambda = \lambda_1 + {\rm i}\lambda_2\in\mathbb C\setminus [-\alpha_{n}, \alpha_{n}]$ by the Cauchy-Riemann equations
\begin{equation*}
\partial_{\lambda_1} a_n(\lambda_1, \lambda_2)= \partial_{\lambda_2} b_n(\lambda_1, \lambda_2), 
\quad \partial_{\lambda_2} a_n(\lambda_1, \lambda_2)= -\partial_{\lambda_1} b_n(\lambda_1, \lambda_2).
\end{equation*}

Notice that based on the above analytic extension of $\beta_n(\lambda)$, the imaginary part $\Im\beta_n(\lambda)$ is negative in $\mathbb C^-$ now. Therefore, for large $\alpha_n$
the kernel of $\widetilde{R}_0(\beta_n(\lambda))$ satisfies
\begin{equation}\label{H01_asymptotic}
H_0^{(1)}(\beta_n(\lambda)|\tilde{x} - \tilde{y}|) \sim \mathrm{e}^{\alpha_n |\tilde{x} - \tilde{y}|},
\end{equation}
Throughout, $H_m^{(1)}$ denotes the Hankel function of the first kind of order $m$. The asymptotic behavior \eqref{H01_asymptotic} implies that the series in \eqref{sum1} may not converge. To resolve this issue and study the analytic continuation of the resolvent $R_V(\lambda)$, we shall make the following assumption:

\begin{itemize}
	\item[(A)] The source function $f(x)$ has only a finite number of modes defined in \eqref{fourier}, i.e., there exits a positive integer $N_0>0$ such that 	  \begin{equation*}
		f(x) = \sum_{n = 1}^{N_0} f_n(\tilde{x})\sin(\alpha_n x_3). 
	\end{equation*}
\end{itemize}

\subsection{Resolvent estimates}

Based on the preparations in the previous subsection, we are now ready to establish several crucial estimates on the resolvent.
We begin with a few notations. For $\theta\in (0, \frac{\pi}{2})$, denote the sectorial domain $S_\theta$ by
\[
S_\theta = \{z\in\mathbb C: \arg z\in (-\theta, \theta) \cup (\pi-\theta, \pi+\theta), \ z\neq 0 \}.
\]
Hereafter, we denote by $t_{-}:=\max\{-t,0\}$ and ${\rm diam}({\rm supp}\rho): = \sup\{|\tilde{x} - \tilde{y}| : \tilde{x}, \tilde{y} \in {\rm supp}\rho\}$ for $\rho\in C_0^\infty(\mathbb R^2)$. We also sometimes simplify the relation ``$a\le C b$'' as ``$a\lesssim b$'' with a nonessential constant $C>0$ that might differ at each occurrence.

The following proposition concerns the analytic continuation of the free resolvent 
$\widetilde{R}_V(z) = (-\Delta_{\tilde{x}} - z^2 + V)^{-1}$ in two dimensions. 

\begin{proposition}{\cite[Theorem 3]{ZZ}}\label{meromorphic}
	Let $\rho\in C_0^\infty(\mathbb R^2)$ with $\rho = 1$ on $\mathrm{supp}V$. The resolvent $\rho \widetilde{R}_V(z)\rho$ is a meromorphic family of operators in $S_\theta$. Moreover, $\rho \widetilde{R}_V(z)\rho$ is analytic for $z\in S_\theta\cap \Omega_\delta$
	with the following resolvent estimates
	\begin{align*}
		\|\rho \widetilde{R}_V(z)\rho\|_{L^2(\mathbb R^2)\rightarrow H^j(\mathbb R^2)}\lesssim |z|^{-1/2}(1 + |z|^2)^{j/2}
		\mathrm{e}^{L(\Im z)_-},\quad j = 0, 1, 2, 
	\end{align*}
	where $L>{\rm diam}({\rm supp}\rho)$. Here $\Omega_\delta$ is defined as
	\begin{align*}
		\Omega_\delta:=\{z: {\Im}z\geq -\delta {\rm log}(1 + |z|),\ |z|\geq C_0\},
	\end{align*}
	where $C_0$ is a positive constant and $\delta$ satisfies $0<\delta<\frac{1}{4L}$.  In particular, there are only finitely many poles in the domain
	\[
	\{z\in\mathbb{C}: {\Im}z\geq -\delta {\rm log}(1 + |z|),\ |z|\leq C_0\}\cap S_\theta.
	\]
\end{proposition}

We represent the solution to \eqref{eqn_n}--\eqref{src} by $u_n = \widetilde{R}_V( \beta_n(\lambda))f_n$. Let $M$ and $h$ be positive constants such that $M>\alpha_{N_0}$, and denote the strip $\mathcal  R $ by $\mathcal  R := [M, +\infty)\times [-h, h]$. We next show that $ \beta_n(\lambda)\in S_\theta\cap\Omega_\delta$ which is complex analytic for $\lambda\in\mathcal R\cup - \mathcal R$ with appropriate choices of $h$ and $M$. From the analytic extension discussed above, $ \beta_n(\lambda)$ is complex analytic in $\mathbb C\setminus [-\alpha_{N_0}, \alpha_{N_0}]$.
For $h$ being sufficiently small and $M$ being sufficiently large, we have $\gamma_n>0$ in \eqref{ge} and then
$|\Im\beta_n| \leq |\Re\beta_n|$ for $\lambda\in\mathcal R\cup-\mathcal R$ which yields $\beta_n\in S_\theta\cap\Omega_\delta$. We can also have $ |\Re\beta_n|>C_0$ where $C_0$ is specified in Proposition \ref{meromorphic}. Moreover, the following estimate holds
\[
|\Im \beta_n(\lambda)| = \left(\frac{\sqrt{\gamma_n^2 + \eta^2} - \gamma_n}{2}\right)^{1/2} \lesssim h|\lambda_1|.
\] 
As a consequence, the resolvent $\widetilde{R}_V( \beta_n(\lambda)): L_{\rm comp}^2(\mathbb R^2)\rightarrow H^j_{\rm loc}(\mathbb R^2)$
is analytic for $\lambda\in\mathcal R\cup - \mathcal R$ with the following resolvent estimates given a fixed $\rho\in C_0^\infty(\mathbb R^2)$
\[
\|\rho \widetilde{R}_V( \beta_n(\lambda))\rho\|_{L^2(\mathbb R^2)\rightarrow H^j(\mathbb R^2)}\le C |\lambda|^{-1/2}(1 + |\lambda|^2)^{j/2}
\mathrm{e}^{C|\lambda|},\quad j = 0, 1, 2.
\]
Noting
\[
R_V(\lambda) = \sum_{n = 1}^{N_0} \widetilde{R}_V( \beta_n(\lambda)) (f_n),
\]
we have the following theorem which provides a resonance-free region and resolvent estimates in the geometry of a planar waveguide.
It will play an important role in the subsequent study of the inverse problem.
\begin{theorem}\label{main_d}
	Let $\eta\in C_0^\infty(D)$ with $\eta = 1$ on $\text{supp}V$
	and let $f$ satisfy Assumption (A). There exist $M>0$ and $h>0$ such that the resolvent $\eta R_V(\lambda)\eta: L^2(D)\rightarrow L^2(D)$ is an analytic family of operators for $\lambda\in\mathcal R \cup - \mathcal R$ with the following resolvent estimates
	\begin{align*}
		\|\eta R_V(\lambda)\eta f\|_{L^2(D)\rightarrow H^j(D)}\leq C|\lambda|^{-1/2}(1 + |\lambda|^2)^{j/2} \mathrm{e}^{C|\lambda|} \|f\|_{L^2(D)},\quad j = 0, 1, 2, 
	\end{align*}
	where $C$ is a positive constant.
\end{theorem}

In what follows, we consider the case that $V = 0$. More precisely, we investigate the free resolvent $R_0(\lambda) = (-\Delta - \lambda^2 )^{-1}$. We will see that in this case, the free resolvent has a larger analytic domain.

The following proposition concerns the analytic continuation of the free resolvent $\widetilde{R}_0(z) = (-\Delta_{\tilde{x}}  - z^2)^{-1}$ in $\mathbb{R}^2$.
\begin{proposition}{\cite[Theorem 10]{ZZ}}
	The free resolvent $\widetilde{R}_0(z)$ is analytic for $z\in S_\theta, \Im z>0$ as a family of operators
	\begin{align*}
		\widetilde{R}_0 (z): L^2(\mathbb R^{2})\to L^2(\mathbb R^{2})
	\end{align*}
	where $\|\widetilde{R}_0 (z)\|_{L^2(\mathbb R^{2})\to L^2(\mathbb R^{2})} = \mathcal{O}(1/|z|^{1/2})$.
	Moreover, for each $\rho\in C_0^\infty(\mathbb R^{2})$ the free resolvent $R_0(\lambda)$ extends to a family of analytic operators for $z\in S_\theta$ as follows
	\begin{align*}
		\rho \widetilde{R}_0 (z) \rho: L^2(\mathbb R^{2})\to L^2(\mathbb R^{2})
	\end{align*}
	with the resolvent estimates
	\begin{align*}
		\|\rho \widetilde{R}_0(z) \rho\|_{L^2(\mathbb R^2)\rightarrow H^j(\mathbb R^2)}\lesssim |z|^{-1/2} (1+|z|^2)^{j/2} \mathrm{e}^{L (\Im z)_-}, \quad j=0, 1, 2,
	\end{align*}
	where $L>{\rm diam}({\rm supp}\rho)$.
\end{proposition}

We denote the solution to \eqref{eqn_n}--\eqref{src}  when $V=0$ by $u_n(\tilde{x}, \lambda) = \widetilde{R}_0( \beta_n(\lambda))f_n(\tilde{x})$. 
We next prove that the resolvent $\widetilde{R}_0( \beta_n(\lambda)): L^2_{\rm comp}(\mathbb R^2)\to L^2_{\rm loc}(\mathbb R^2)$ 
is an analytic family of bounded operators for $\lambda\in S_{\theta}\setminus [-\alpha_{N_0}, \alpha_{N_0}]$.
To achieve this goal, it suffices to show that the range of $\beta_n(\lambda)$ is contained in some sectorial domain $S_{\theta_1}$ for 
$\lambda\in S_{\theta}\setminus [-\alpha_{N_0}, \alpha_{N_0}]$. 
Indeed, from the analytic extension discussed above, $ \beta_n(\lambda)$ is complex analytic in $\mathbb C\setminus [-\alpha_{N_0}, \alpha_{N_0}]$.
As a result, we have that $\widetilde{R}_0( \beta_n(\lambda))$ is analytic for $\lambda\in S_{\theta}\setminus [-\alpha_{N_0}, \alpha_{N_0}]$
since $\widetilde{R}_0(z)$ is analytic for $z\in S_{\theta_1}$.

Since $\lambda\in S\setminus [-\alpha_{N_0}, \alpha_{N_0}] $, 
we have that there exists some $c>0$ such that $|\Im\lambda|\leq c|\Re\lambda|$. Thus, the ratio of $|\Im\beta_n|$ and $|\Re\beta_n|$ satisfies
\[
\overline{\lim}_{|\lambda|\to\infty}\frac{|\Im\lambda|}{|\Re\lambda|} = \overline{\lim}_{|\lambda|\to\infty} \frac{|b_n|}{|a_n|} = c_1, \quad \lambda\in S_{\theta}\setminus [-\alpha_{N_0}, \alpha_{N_0}],
\]
where $c_1$ is a finite positive constant. This implies that the range of $\beta_n(\lambda)$ is contained in some sectorial domain $S_{\theta_1}$.
Therefore, we have that $\widetilde{R}_0( \beta_n(\lambda))$ is analytic for $\lambda\in S_{\theta}\setminus [-\alpha_{N_0}, \alpha_{N_0}]$.
As a consequence, the resolvent $\widetilde{R}_0( \beta_n(\lambda)): L_{\rm comp}^2(\mathbb R^2)\rightarrow H^j_{\rm loc}(\mathbb R^2)$
is also analytic for $\lambda\in S_\theta\setminus [-\alpha_{N_0}, \alpha_{N_0}] $ with the following resolvent estimates given a fixed $\rho\in C_0^\infty(\mathbb R^2)$
\[
\|\rho \widetilde{R}_0( \beta_n(\lambda))\rho\|_{L^2(\mathbb R^2)\rightarrow H^j(\mathbb R^2)}\leq C|\lambda|^{-1/2}(1 + |\lambda|^2)^{j/2}
\mathrm{e}^{C|\lambda|},\quad j = 0, 1, 2.
\]
Noting
\[
R_0(\lambda) = \sum_{n = 1}^{N_0} \widetilde{R}_0( \beta_n(\lambda)) (f_n)\sin\alpha_n x_3,
\]
we have the following theorem for the free resolvent.
\begin{theorem}
	Let $\eta\in C_0^\infty(D)$
	and let $f$ satisfy Assumption (A). The resolvent $\eta R_0(\lambda)\eta: L^2(D)\rightarrow L^2(D)$ is an analytic family of operators for $\lambda\in S_\theta$
	with $\theta\in (0, \frac{\pi}{2})$.
	Moreover, the following resolvent estimates hold
	\begin{align*}
		\|\eta R_0(\lambda)\eta f\|_{L^2(D)\rightarrow H^j(D)}\leq C|\lambda|^{-1/2}(1 + |\lambda|^2)^{j/2} \mathrm{e}^{C|\lambda|} \|f\|_{L^2(D)},\quad j = 0, 1, 2, 
	\end{align*}
	where $C$ is a positive constant. 
\end{theorem}

\section{Inverse problem I: source identification}\label{sec: IP1}

In this section, we investigate the inverse source problem. The uniqueness of the inverse source problem is established given wavenumbers only in a bounded domain.
We also develop a multi-wavenumber numerical scheme to reconstruct the source term from limited aperture Dirichlet boundary measurements. Then several numerical examples are provided to verify the effectiveness of the method. This section ends with an extensional glimpse into the uniqueness in the far-field case.

\subsection{Uniqueness}

In this subsection, we study the unique determination of the source from multi-wavenumber boundary measurements. We further assume that $V\geq 0$ is a real-valued function.

We consider the spectrum of the Schr\"odinger operator $-\Delta_{\tilde{x}} + V$ with the Dirichlet boundary condition in $\widetilde{B}_R$. Specifically, we let $\{\mu_j, \phi_j\}_{j=1}^\infty$ be the positive increasing eigenvalues and eigenfunctions of $-\Delta_{\tilde{x}} + V$ in $\widetilde{B}_R$, where $\phi_j$ and $\mu_j$
satisfy
\[
\begin{cases}
	(-\Delta_{\tilde{x}} + V) \phi_j=\mu_j\phi_j &\quad\text{in }\widetilde{B}_R,\\
	\hspace{17mm}\phi_j=0 &\quad\text{on }\partial \widetilde{B}_R.
\end{cases}
\]

Let $k^2_{n, j} :=\mu_j + \alpha_n^2, 1\leq n\leq N_0, j\in\mathbb N^+$ denote the eigenfrequencies.
To properly formulate the inverse source problem, we additionally require that $k_{n, j}\notin \{\alpha_i\}_{i = 1}^{N_0}$ 
to avoid the resonances. In fact, due to the finiteness of the resonances $\{\alpha_n\}_{n = 1}^N$, this requirement could be fulfilled in certain scenarios. 
For instance, a large width $L$ of the waveguide would narrow down the range of the resonance distribution, such that the resonances are confined in a small vicinity of zero, and thus we may have $k_{n, j}\geq \alpha_{N_0}$.

Assume that $\phi_j$ is normalized such that
\[
\int_{\widetilde{B}_R}|\phi_j(\tilde{x})|^2{\rm d}\tilde{x}=1.
\] 

Since $\{\phi_j(\tilde{x}) \sin\alpha_n x_3\}_{j\in\mathbb N^+, 1\leq n\leq N_0}$ forms an orthogonal basis of the space
\[
\mathcal C_Q := \{f\in L_{\rm comp}^2(D): f \,\, \text{satisfies Assumption (A)}, \, \text{supp}f\subset C_R, \, \|f\|_{L^2(D)}\leq Q\},
\]
the orthogonality leads to the spectral decomposition of $f$:
\[
f(x)=\sum_{n=1}^{N_0}\sum_{j=1}^\infty f_{n, j}\phi_j(\tilde{x}) \sin\alpha_n x_3\ \text{with}\ f_{n, j}=\int_{C_R}f(x)\overline{\phi_j}(\tilde{x}) \sin\alpha_n x_3{\rm d}x.
\]

Let $u(x, k_{n, j})$ be the radiating solution to \eqref{model} with wavenumber $k_{n, j}$.
Noting the boundary conditions \eqref{bc},
multiplying both sides of \eqref{model} by $\overline{\phi_j}(\tilde{x}) \sin\alpha_n x_3$ and integrating by parts over $C_R$ gives
\begin{align}
f_{n, j} &= \int_{\Gamma_R} \left(u(x, k_{n, j}) \partial_{\nu_{\tilde x}}\overline{\phi_j}(\tilde{x}) \sin\alpha_n x_3- \partial_{\nu_{\tilde x}} u(x, k_{n, j}) \overline{\phi_j}(\tilde{x}) \sin\alpha_n x_3\right) {\rm d}\tilde{x}{\rm d}x_3 \notag\\
	&=  \int_{\partial\widetilde{B}_R} u_n(\tilde{x}, k_{n, j}) \partial_{\nu_{\tilde x}}\overline{\phi_j}(\tilde{x}){\rm d}\tilde{x} - \int_{\partial\widetilde{B}_R} \partial_{\nu_{\tilde x}}u_n(\tilde{x}, k_{n, j}) \overline{\phi_j}(\tilde{x})\,{\rm d}\tilde{x}\notag\\
	&=  \int_{\partial\widetilde{B}_R} u_n(\tilde{x}, k_{n, j}) \partial_{\nu_{\tilde x}}\overline{\phi_j}(\tilde{x})\,{\rm d}\tilde{x}, \label{reconstruction}
\end{align}
where
\[
u_n(\tilde{x}, k_{n, j}) = \frac{2}{L}\int_0^L u(x, k_{n, j})\sin\alpha_n x_3{\rm d}x_3.
\]

The following lemma \cite[Lemma A.2]{LZZ} is useful in the subsequent analysis.

\begin{lemma}\label{eigenfunction_est1}
	 Let $\{\mu_j, \phi_j\}_{j=1}^\infty$ be the eigensystem of $-\Delta_{\tilde{x}} + V$ in $\widetilde{B}_R$. Then it holds
	\begin{equation*}
		\|\partial_{\nu_{\tilde{x}}} \phi_j\|^2_{L^2(\partial \widetilde{B}_R)}\leq C\mu_j,
	\end{equation*}
	where the positive constant $C$ is independent of $j$.
\end{lemma}

If given $u(x, k)$ at all eigenfrequencies $\{k_{n, j}, 1\leq n\leq N_0, j\in\mathbb N^+\}$, we have the following Lipschitz stability estimate.
\begin{theorem}
	The following stability estimate holds:
	\[
	\|f\|^2_{L^2(C_R)}\lesssim \sum_{n=1}^{N_0}\sum_{j=1}^\infty k_{n, j}\|u_n(\tilde{x}, k_{n, j})\|^2_{L^2(\partial\widetilde{B}_R)}.
	\]
\end{theorem} 

\begin{proof}
	By the integral identity \eqref{reconstruction} and Lemma \ref{eigenfunction_est1} we have
	\[
	|f_{n, j}|^2 \lesssim k_{n, j} \|u_n(\tilde{x}, k_{n, j})\|^2_{L^2(\partial\widetilde{B}_R)},
	\]
	which completes the proof with the aid of the energy relation
        \[
	\|f\|^2_{L^2(C_R)} = \frac{L}{2}\sum_{n=1}^{N_0}\sum_{j=1}^\infty |f_{n, j}|^2.
        \]
\end{proof}

The above Lipschitz stability estimate implies the uniqueness of the inverse source problem, but it requires the boundary data at all $k_{n, j}$. In practical applications, however, it is more reasonable to collect data only at wavenumbers in a finite interval.
By analyticity of the data proved in Theorem \ref{main_d}, we establish the following uniqueness result for the inverse source problem with wavenumbers in a finite interval. 

\begin{theorem}\label{uisp}
	Let $f\in L^2(C_R)$ and $I:=(M, M + \delta) \subset \mathbb R^+$ be an
	open interval, where $M$ is the constant specified in Theorem \ref{main_d} and $\delta$ is any positive
	constant. Then the source term $f$ can be uniquely determined by the
	multi-frequency data $\{u(x, k): x\in \Gamma_R,  k \in I\}\cup
	\{u(x, k_{n, j}): x\in\Gamma_R,  k_{n, j}^2\leq M\}$. 
\end{theorem}

\begin{proof}
	Let $u(x, k)=0$ for $x\in\Gamma_R$ and $ k \in I \cup \{ k_j:
	j\in\alpha\}$. It suffices to show that $f(x)=0$. Since $u(x, k)$
	is analytic in $\mathcal R$ for $x\in\Gamma_R$, it holds that
	$u(x, k) = 0$ for all $ k^2>M$. As the data  $\{u(x, k_j): x\in\Gamma_R,  k_{n, j}^2\leq M\}$ is also available,
	we have that
	\[
	\int_{\partial\widetilde{B}_R} u_n(\tilde{x}) \partial_{\nu_{\tilde x}}\overline{\phi_j}\,{\rm d}\tilde{x} = 0
	\]
	for all $j\in\mathbb N^+$ and $1\leq n\leq N_0$.
	It follows from \eqref{reconstruction} that 
	\[
	f_{n, j} = 0, \quad j\in\mathbb N^+, \, 1\leq n\leq N_0,
	\]
	which implies $f = 0$.
\end{proof}

\begin{remark}
	Motivated by \eqref{reconstruction} and the Fourier method in \cite{ZG}, we aim to approximate the unknown source function by a finite Fourier expansion of the form
	\[
	f_N(\tilde{x}, x_3) = \sum_{n=1}^N \sum_{j=1}^J f_{n, j} \phi_j(\tilde{x}) \sin\alpha_n x_3, \quad N, J\in\mathbb N^+.
	\]
	For numerics, Assumption (A) is not necessary.
\end{remark}

\subsection{Reconstruction algorithm}

This subsection deals with the numerical scheme for the null-potential inverse source problem, i.e., recover $f$ in the case $V=0$. Now the radiated field $u$ satisfies the Helmholtz equation
\begin{equation}\label{eq: Helmholtz}
	-\Delta u-k^2u=f,\quad \text{in}\ D
\end{equation}
with $k\in\mathbb{R}^+.$  Let $\{k_j\}_j$ be a finite number of frequencies, then the multi-frequency inverse source problem under consideration in this subsection is to reconstruct the source function $f(x)$ in \eqref{eq: Helmholtz} from the measurement $\{u(x; k):\,x\in \Gamma_R,\ k\in\{k_j\}_j\}.$

In terms of the series expansions \eqref{eq:seriesExpansion}, it is readily seen that each Fourier mode $u_n(\widetilde{x})$ satisfies the Sommerfeld radiated condition \eqref{src} and
\begin{align}\label{eq:unHelmholtz}
	\Delta_{\widetilde{x}}u_n+\beta_n^2(k) u_n=-f_n,
\end{align}
where $f_n$ is the Fourier mode of $f$ given by \eqref{fourier}, and $\beta_n(k)=\sqrt{k^2-\alpha_n^2}$ with $\Im\beta_n(k)\ge 0$.
In addition, the assumption that $\mathrm{supp} f\subset C_R$, without loss of generality, implies that there exists $V_0=(-a/2, a/2)^2 (a>0)$ such that $\mathrm{supp} f_n \subset V_0\subset \widetilde{B}_R.$

Though the modal wavenumber $\beta_n(k)$ is determined by $k$ and $\alpha_n$, from another point of view, once $\alpha_n$ and  $\beta$ are given in advance, the wavenumber $k$ can be obtained correspondingly through $k_{n,\beta}=\sqrt{\alpha_n^2+\beta^2}$. This further makes it possible to choose $\beta$ flexibly and it is unnecessary to relate the modal wavenumber $\beta$ with $n$ explicitly. In this view, \eqref{eq:unHelmholtz} can be rewritten as 
\begin{equation*}
	\Delta_{\widetilde{x}}u_n(\widetilde{x})+\beta^2u_n(\widetilde{x})=-f_n(\widetilde{x}).
\end{equation*}
Once the mode $f_n$ is recovered, the source function $f$ is correspondingly determined uniquely through \eqref{eq:seriesExpansion}. Thus, in the rest of this subsection, we aim to recover $f_n$ for $n=1,2,\cdots$  from $\{u_n(x;\beta):x\in\Gamma_R,\beta\in\mathcal{B}_N\}$ with $\mathcal{B}_N$ being the set of admissible modal wavenumbers, which is defined by 
\begin{definition}[Admissible modal wavenumbers]
	Let $N\in\mathbb{N}_+,$ and $\beta^*\in\mathbb{R}_+$ be a small modal wavenumber such that $0<\beta^*R<1.$ Then the admissible set of modal wavenumbers is given by 
	\begin{align*}
		\mathcal{B}_N=\left\{\frac{\pi}{a}|\bm\ell|:{\bm\ell}\in\mathbb{Z}^2:1\le|\bm\ell|_\infty\le N\right\}\cup\{\beta^*\}.
	\end{align*}
\end{definition}

The basic idea is to approximate $f_n$ by the following Fourier expansion 
\begin{align}\label{eq:fn}
	f_n^N(\widetilde{x})=\sum_{|\bm\ell|\le N}\widehat{f}_{{\bm\ell},n}\phi_{\bm\ell}(\widetilde{x})\ \text{with}\ \phi_{\bm\ell}(\widetilde{x})=\mathrm{e}^{\mathrm{i}\frac{2\pi}{a}{\bm\ell}\cdot \widetilde{x}},\quad  n=1,2,\cdots,
\end{align}
where ${\bm\ell}=(\ell_1,\ell_2)\in\mathbb{Z}^2$ and $\phi_{\bm\ell}$ are the Fourier basis functions. In \cref{eq:fn}, $\widehat{f}_{{\bm\ell},n}$ are the Fourier coefficients corresponding to the index $\bm\ell$ and the $n$-th Fourier mode $f_n$, 
\begin{align*}
	\widehat{f}_{{\bm\ell},n}=\frac{1}{a^2}\int_{V_0}f_n(\widetilde{x})\overline{\phi_{\bm\ell}}(\widetilde{x})\mathrm{d}\widetilde{x},\quad n=1,2,\cdots.
\end{align*}

Let $\nu_\rho$ be the unit outward normal to $\Gamma_\rho:=\{\widetilde{x}\in\mathbb{R}^2:|\widetilde{x}|=\rho>R\}.$ Define
\begin{align}\label{eq:2.6}
	w_n(\widetilde{x};\beta)&:=\sum_{m\in\mathbb{Z}}\frac{H_m^{(1)}(\beta\rho)}{H_m^{(1)}(\beta R)}\widehat{u}_{\beta, n, m}\mathrm{e}^{\mathrm{i}m\theta},\quad\widetilde{x}\in\Gamma_\rho,\quad n=1,2,\cdots,\\
	\partial_{\nu_\rho}w_n(\widetilde{x};\beta)&:=\sum_{m\in\mathbb{Z}}\beta\frac{{H_m^{(1)}}'(\beta\rho)}{H_m^{(1)}(\beta\label{eq:2.7} R)}\widehat{u}_{\beta, n, m}\mathrm{e}^{\mathrm{i}m\theta},\quad\widetilde{x}\in\Gamma_\rho,\quad n=1,2,\cdots,
\end{align}
where $\beta\in\mathcal{B}_N,$ and 
\begin{align*}
	\widehat{u}_{\beta, n, m}=\frac{1}{2\pi}\int_0^{2\pi}u_n(R,\theta;\beta)\mathrm{e}^{-\mathrm{i}m\theta}\,\mathrm{d}\theta.
\end{align*}

Following \cite{ZG}, the Fourier coefficients $\widehat{f}_{{\bm\ell},n}$ are explicitly given by
\begin{align}\label{eq:fl}
	\widehat{f}_{{\bm\ell},n} & =\frac{1}{a^2}\int_{\Gamma_\rho}\left(
	\partial_{\nu_\rho}w_n(\widetilde{x};\beta)+\mathrm{i}\frac{2\pi}{a}({\bm\ell}\cdot\nu_\rho)w_n(\widetilde{x};\beta)
	\right)\overline{\phi}_{\bm\ell}(\widetilde{x})\,\mathrm{d}s_{\widetilde{x}},\,1\le|\bm\ell|_\infty\le N,\\
	\widehat{f}_{{\bm 0},n}&=\nonumber\frac{\vartheta\pi}{a^2\sin\vartheta\pi}\int_{\Gamma_\rho}\left(\partial_{\nu_\rho}w_n(\widetilde{  x};\beta^*)+\mathrm{i}\frac{2\pi}{a}({\bm\ell}^*\cdot\nu_\rho)w_n(\widetilde{  x};\beta^*)\right)\overline{\phi}_{{\bm\ell}^*}(\widetilde{x})\,\mathrm{d}s_{\widetilde{x}}\\
	&\quad -\frac{\vartheta\pi}{a^2\sin\vartheta\pi}\sum_{1\le|\bm\ell|_\infty\le N}\widehat{f}_{{\bm\ell},n}
	\int_{V_0}\phi_{\bm\ell}(\widetilde{ x})\overline{\phi}_{{\bm\ell}^*}(\widetilde{x})\,\mathrm{d}\widetilde{x}, \label{eq:f0}
\end{align}
where $\vartheta$ is  a constant such that $0<\vartheta<\frac{a}{2\pi},\,{\bm\ell}^*=(\vartheta, 0)$ and $\beta^*=\frac{\pi\vartheta}{a}.$ Hence,    
\begin{align}\label{eq:fN}
	f_N(\widetilde{x}, x_3)=\frac{2}{L}\sum_{n=1}^\infty f_n^N(\widetilde{  x})\sin(\alpha_n x_3),
\end{align}
can be taken as the reconstruction to $f(x)$. We finally remark that following \cite{ZG}, the stability of the Fourier method can be analogously deduced.

\subsection{Numerical verification}

We shall conduct several numerical experiments to verify the performance of the Fourier method proposed for the inverse source problem arising in the waveguide.
The synthetic radiated fields were generated by solving the forward problem via direct integration. Utilizing the Green's function
\begin{equation}\label{eq:G}
G(x, y)  =  \frac{\rm i}{2L}\sum_{n = 1}^{N_0} H_0^{(1)} (\beta_n(k) |\tilde{x} - \tilde{y}|) \sin(\alpha_n y_3)  \sin(\alpha_n x_3),\quad x, y\in D,\ \tilde{x}\neq \tilde{y},
\end{equation}
the radiating solution to the Helmholtz equation \eqref{eq: Helmholtz} is given by
$$
	u(x; k)=-\int_{V_1}G(x, y)f(y)\mathrm{d}y,
$$
where $V_1=V_0\times[0,L]$. The Gauss quadrature is adopted to calculate the volume integrals over the $50^3$ Gauss-Legendre points. Then the synthetic data is corrupted by artificial noise via $u^\delta:=u+\delta r_1|u|\mathrm{e}^{\mathrm{i}\pi r_2}$ where $r_1$ and $r_2$ are two uniformly distributed random numbers ranging from $-1$ to 1, $\delta$ is the noise level.

Now, we specify the details of the implementational aspects of the Fourier method. Let $V_0=[-0.3,0.3]^2, \, L=2.$ We aim to reconstruct the true source $S(x),x\in V_1$ by the Fourier expansion $S_N(x),\,x\in V_1.$ Throughout our numerical experiment, the Green's function \eqref{eq:G} is numerically truncated by $N_0=80.$
Further, given $N\in\mathbb{N}^+,$ the modal wavenumbers are set to be 
\begin{equation*}
	\mathcal{B}_N=\left\{\frac{10\pi}{3}|\bm\ell|:{\bm\ell}\in\mathbb{Z}^2,1\le|\bm\ell|_\infty\le N\right\}\bigcup\left\{\frac{10^{-2}\pi}{3}\right\}.
\end{equation*} 
The radiated data 
\[
\begin{aligned}
	\Bigg\{
	u\left(R,\theta_j, z_\xi;\sqrt{\alpha_n^2+\beta_{\bm\ell}^2}\right):\ &\theta_j=\frac{j\pi}{150},j=1,\cdots, 300; z_\xi=\frac{\xi L}{40},\,\xi=1,\cdots, 40; \\
	&\alpha_n=\frac{2n-1}{2L}\pi,\,n=0,1,\cdots, 40;\beta_{\bm\ell}\in\mathcal{B}_N
	\Bigg\}
\end{aligned}
\]
are measured on $\Gamma_R$ with $R=0.5$.  For a visualization of the measurement surface, we refer to \Cref{fig: GammaR} where the red points denote the receivers.

\begin{figure}
	\centering
	\includegraphics[width=0.3\textwidth]{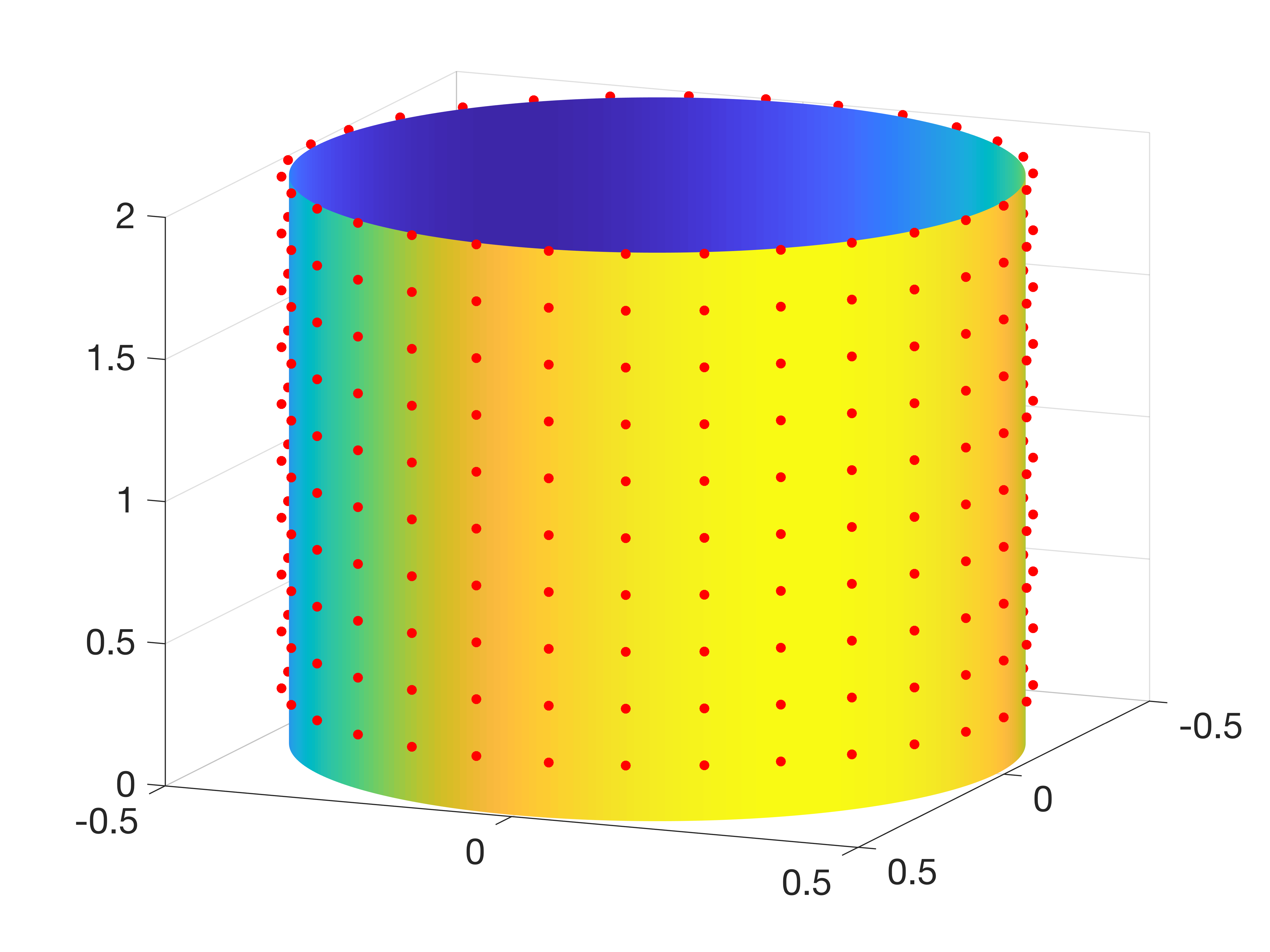}
	\caption{Illustration of the measurement surface.}\label{fig: GammaR}
	\label{fig: receiver}
\end{figure}

Once the radiated field is measured, the corresponding modes of the radiated data can be written as 
\[
\left\{
u_n\left(R,\theta_j;{\beta_{\bm\ell}}\right):\theta_j=\frac{j\pi}{150},j=1,\cdots, 300,\beta_{\bm\ell}\in\mathcal{B}_N
\right\},\quad n=1,2,\cdots.
\]
Next, in terms of  \eqref{eq:2.6}--\eqref{eq:2.7} with $\rho=0.6$ and truncation $|m|\le 70$, we compute the following artificial Cauchy data of the Fourier modes
\[
\left\{ u_n(\rho,\Theta_j;\beta_{\bm\ell}), \partial_{\nu_\rho} u_n(\rho,\Theta_j;\beta_{\bm\ell}):\Theta_j=\frac{j\pi}{400}, j=1,\cdots, 800,\beta_{\bm\ell}\in\mathcal{B}_N
\right\},\ n=1,2,\cdots. 
\]

To calculate the Fourier coefficients $\widehat{f}_{{\bm\ell}, n},$ and $\widehat{f}_{{\bm 0}, n}$ in \eqref{eq:fl} and \eqref{eq:f0}, respectively, we use the trapezoidal rule to evaluate the surface integrals over $\Gamma_\rho$ and the volume integral over $V_0$ is evaluated over a $200\times200$ grid of uniformly spaced points $x_m\in V_0,m=1,\cdots, 200^2.$ Finally, we compute the point-wise values $f_N(\widetilde{x}, x_3)$ by \eqref{eq:fN} with $f_n^N(\widetilde{x})$ determined by \eqref{eq:fn}.

\begin{figure}[!h]
	\centering
	\includegraphics[width=0.4\textwidth]{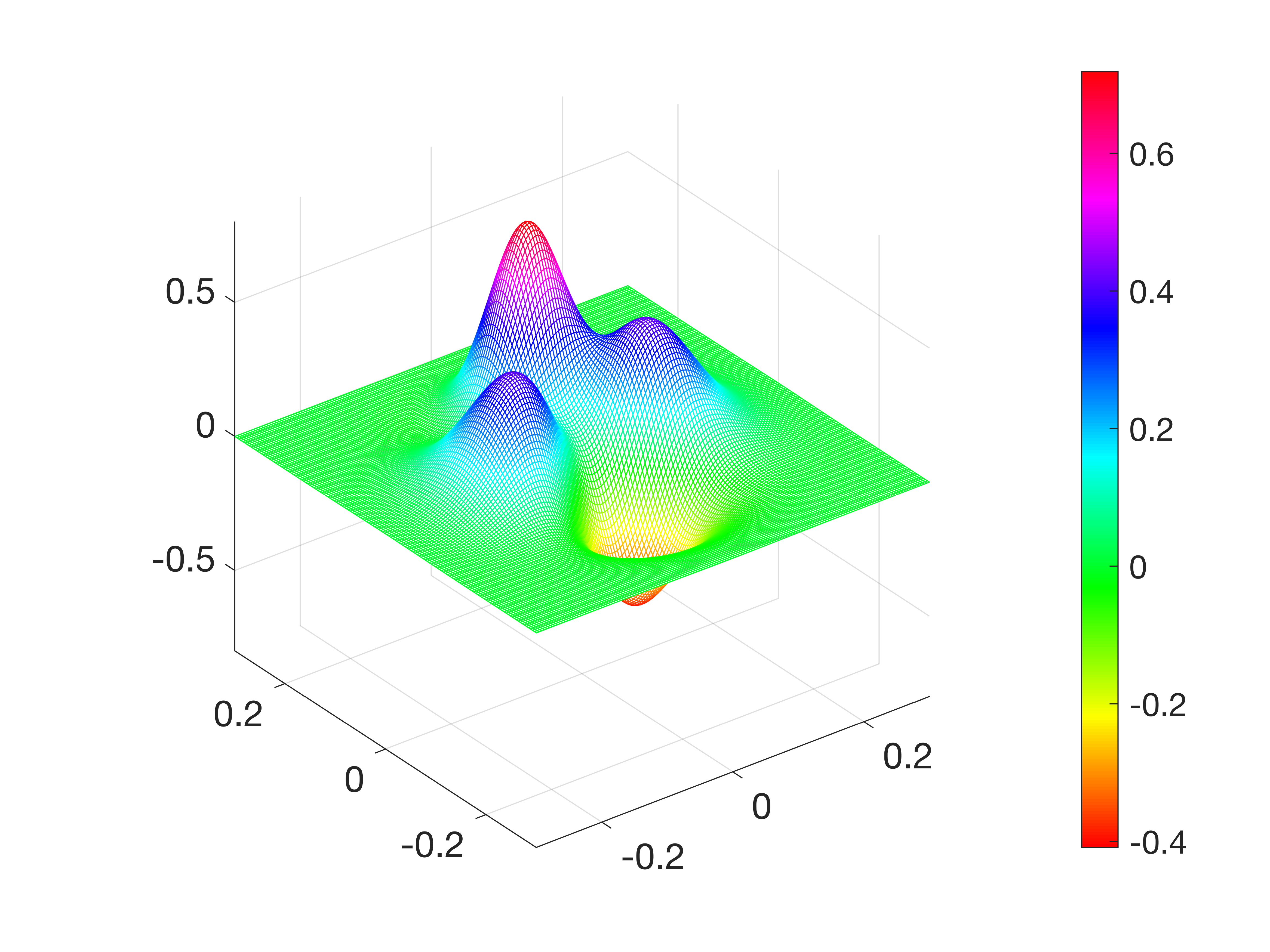}\quad
	\includegraphics[width=0.4\textwidth]{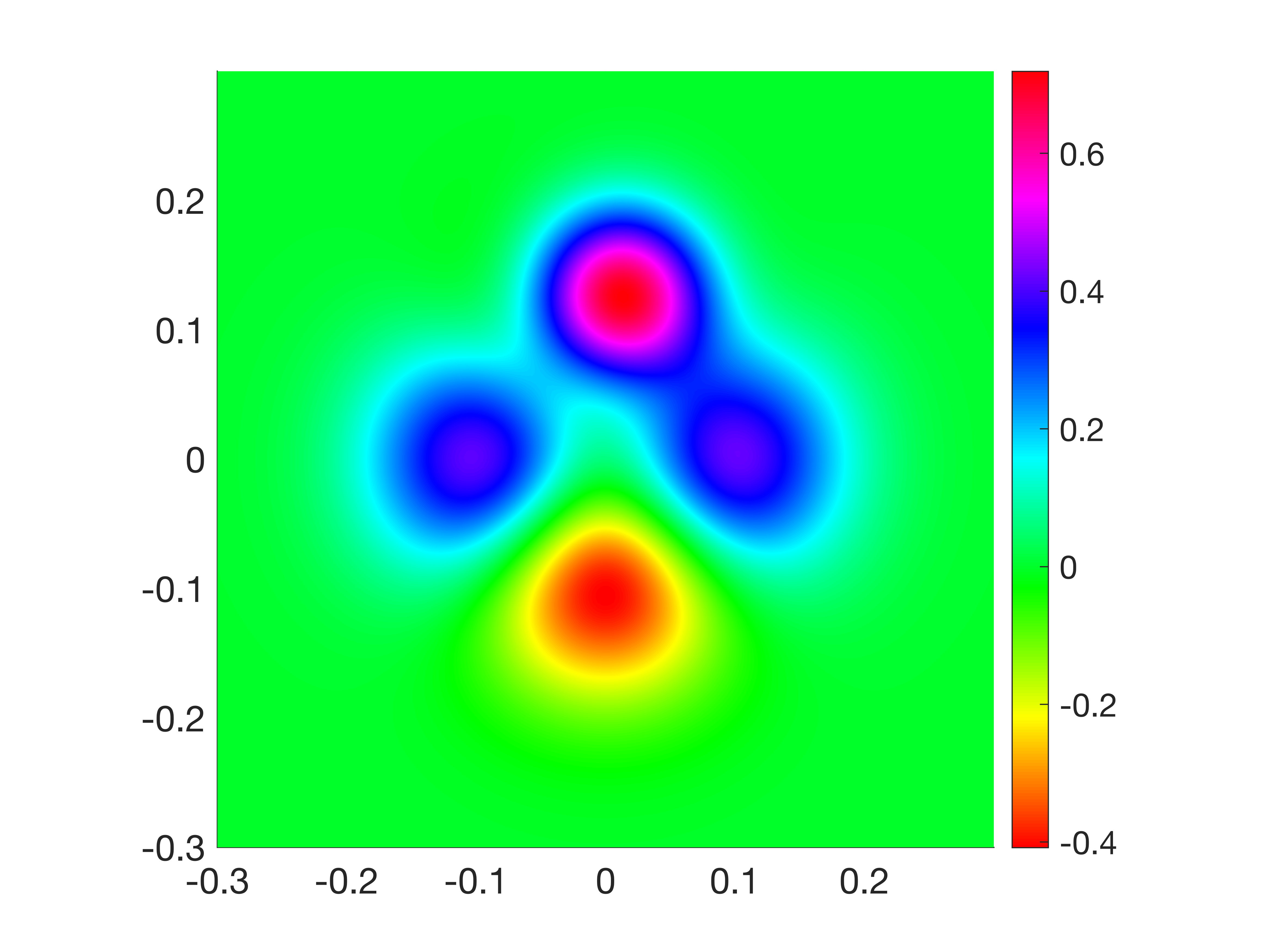}
	\caption{The exact source mode $\widetilde{f}$. Left: surface plot; Right: contour plot.}\label{fig: S1e}
\end{figure}

\begin{example}
	In the first example, we test the performance of the method by considering the reconstruction of the source function which has finite Fourier modes. Specifically, the source function is given by 
	\[
	f_1(x_1,x_2,x_3)=\frac{2}{L}\sum_{n=1}^{40}\widetilde{f}(x_1,x_2)\sin\alpha_n x_3,
	\] 
	with
	\[
	\widetilde{f}(x_1,x_2)=1.1\mathrm{e}^{-200\left((x_1-0.01)^2+(x_2-0.12)^2\right)}
	-100\left(x_2^2-x_1^2\right)\mathrm{e}^{-90\left(x_1^2+x_2^2\right)}.
	\]
\end{example}
In this example, we set $\delta=5\%$ and $N=4$. The $n$-th Fourier mode $f_n$ of the exact source function is independent of $n$ and $f_n(x_1, x_2)\equiv \widetilde{f}(x_1, x_2),\,n=1,2,\cdots$. We refer to \Cref{fig: S1e} for a display of $\widetilde{f}$. 
Theoretically, once a Fourier mode is recovered, the source function can be approximated by multiplying the factor $\sum_{n=1}^{40}\sin\alpha_n x_3$. Nevertheless, we have yet to determine the precise expression of the exact source. Thus, for each $n=1,2,\cdots, 40,$ we reconstruct the Fourier modes $\widetilde{f}_n(x_1, x_2)$ and exhibit the reconstruction of several different Fourier modes in \Cref{fig:S1}.  The results in \Cref{fig: S1e} and \Cref{fig:S1} demonstrate that all the Fourier modes are well recovered, and the reconstruction error is around $4.53\%$. 

\begin{figure}
	\centering
	\includegraphics[width=0.3\textwidth]{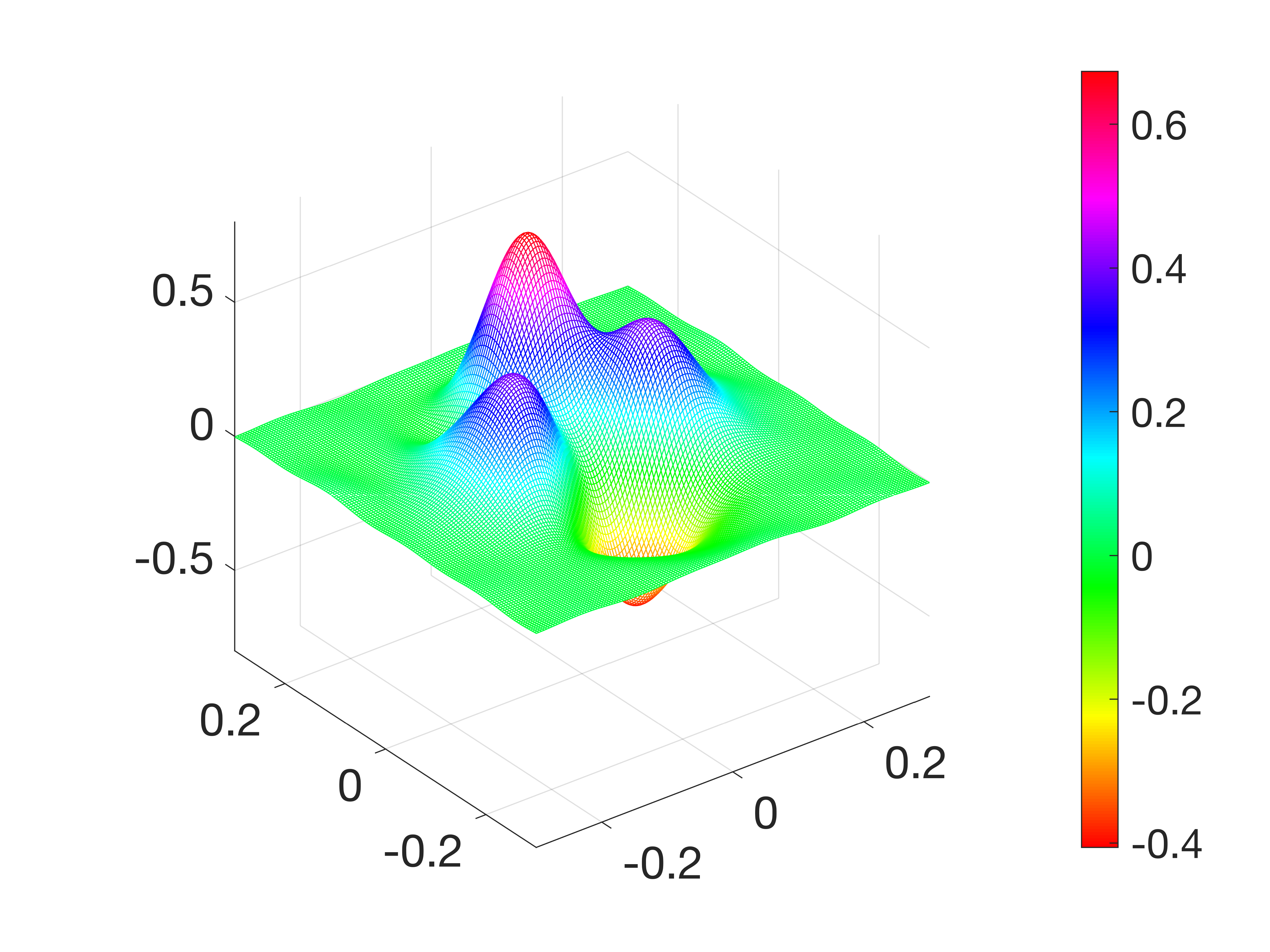}\quad
	\includegraphics[width=0.3\textwidth]{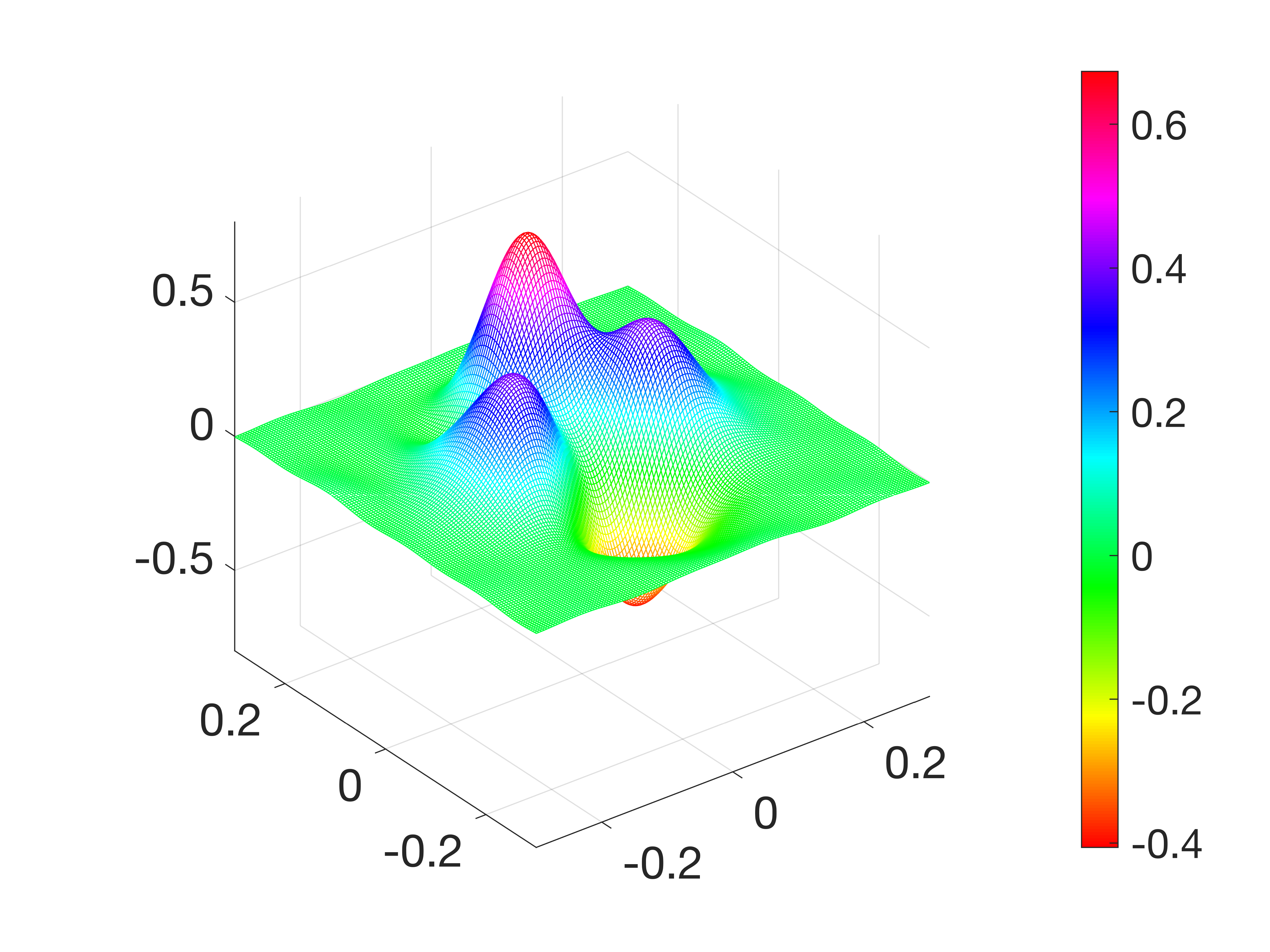}\quad
	\includegraphics[width=0.3\textwidth]{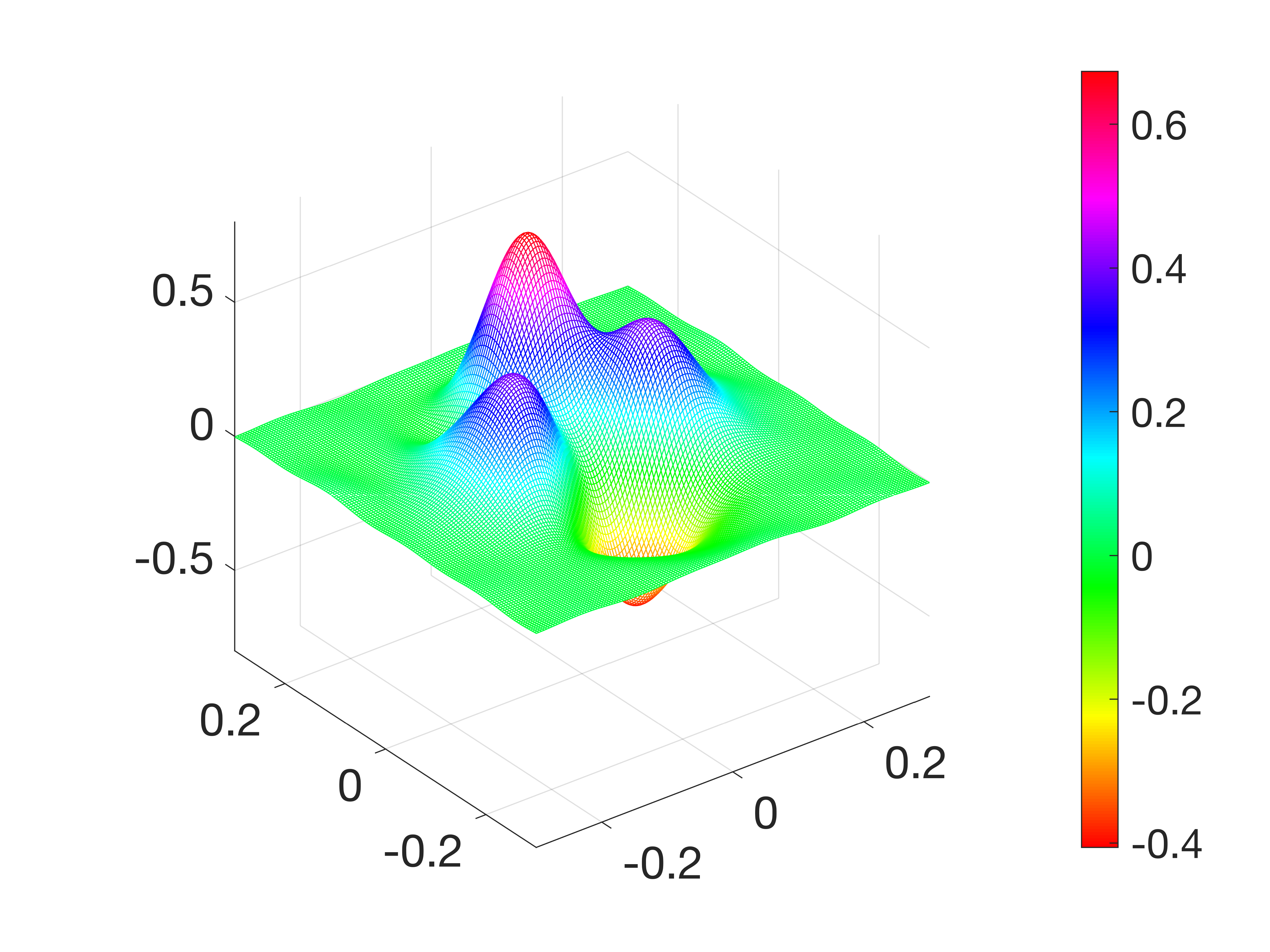}\\
	\includegraphics[width=0.3\textwidth]{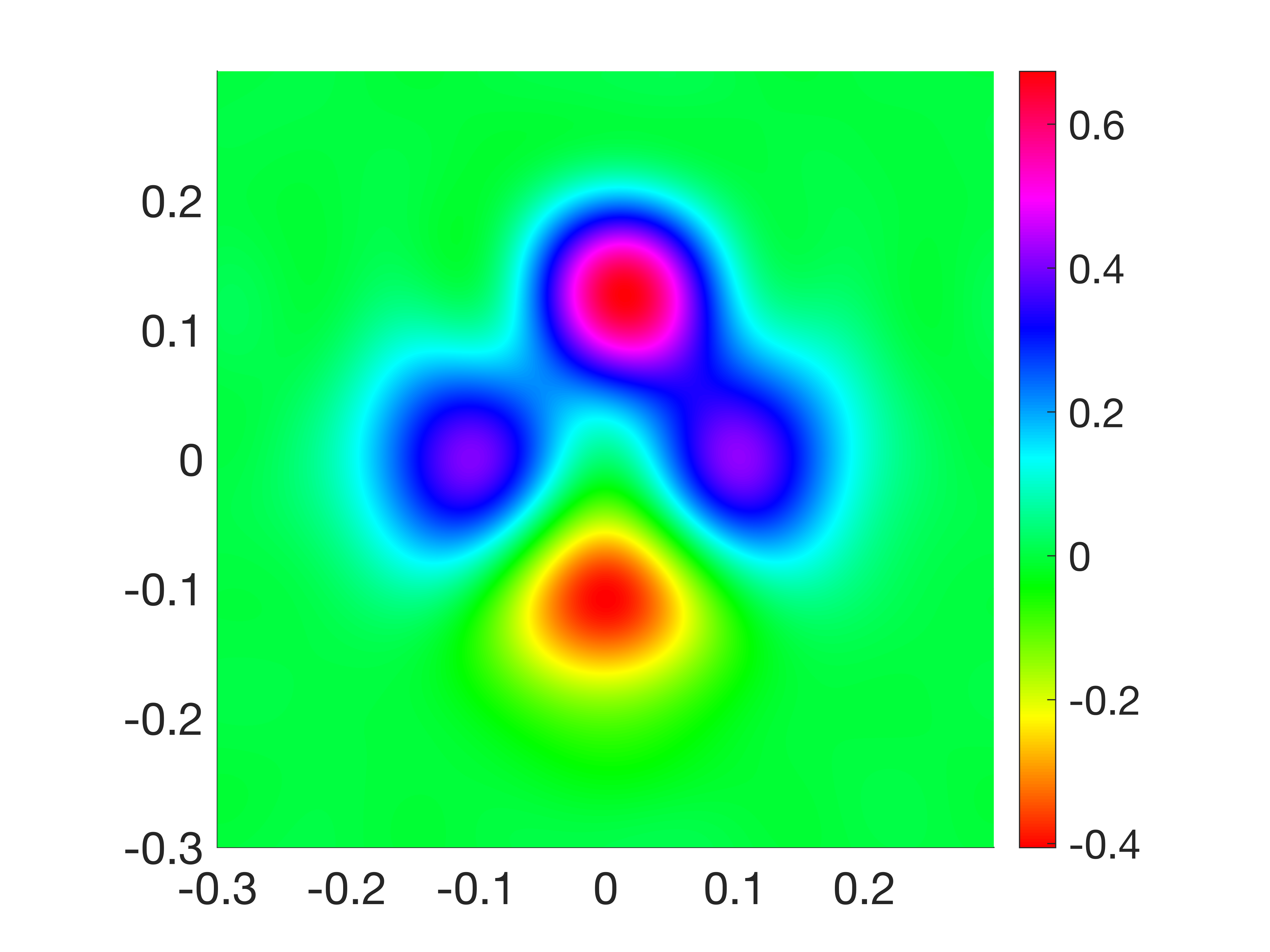}\quad
	\includegraphics[width=0.3\textwidth]{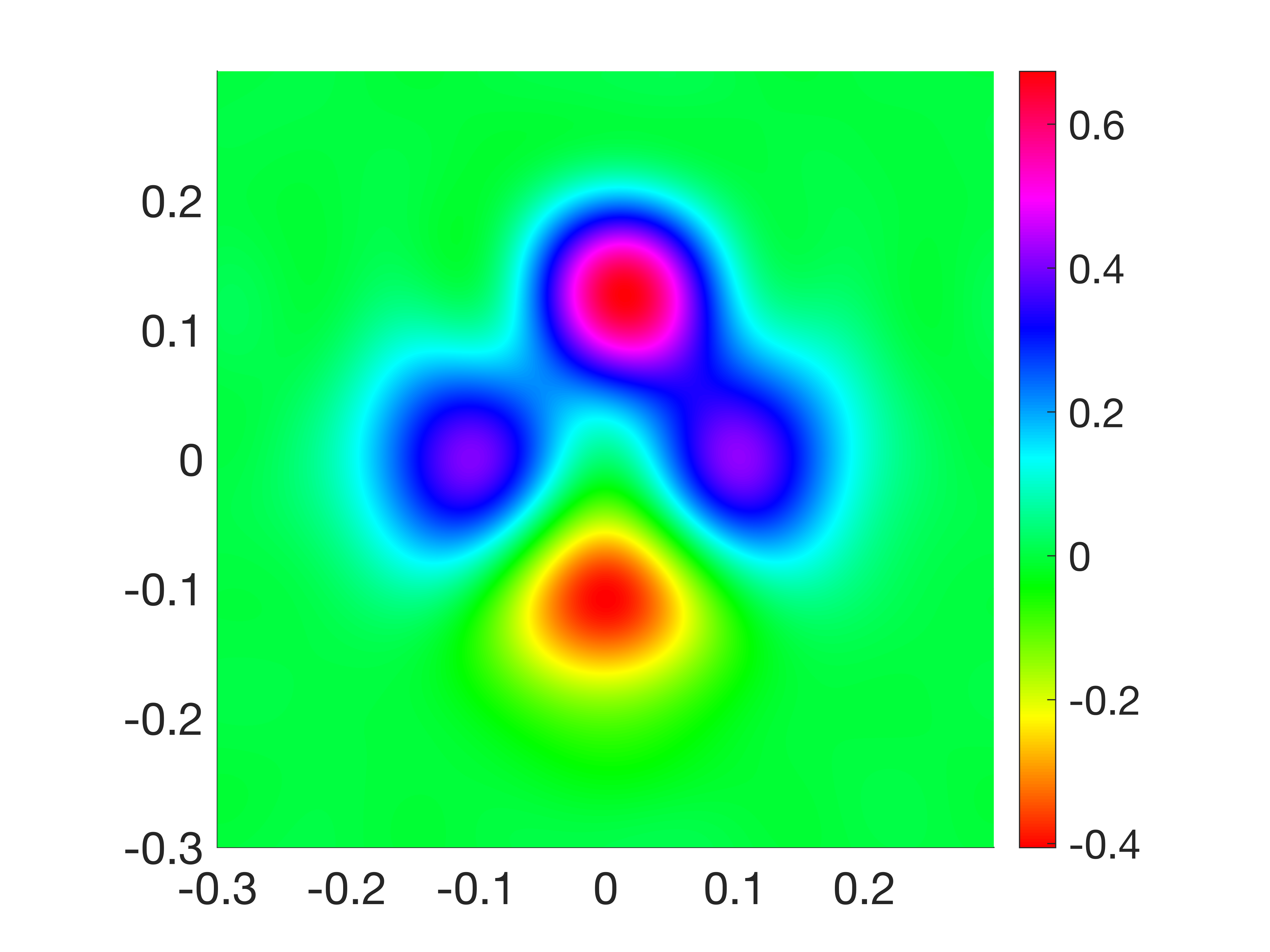}\quad
	\includegraphics[width=0.3\textwidth]{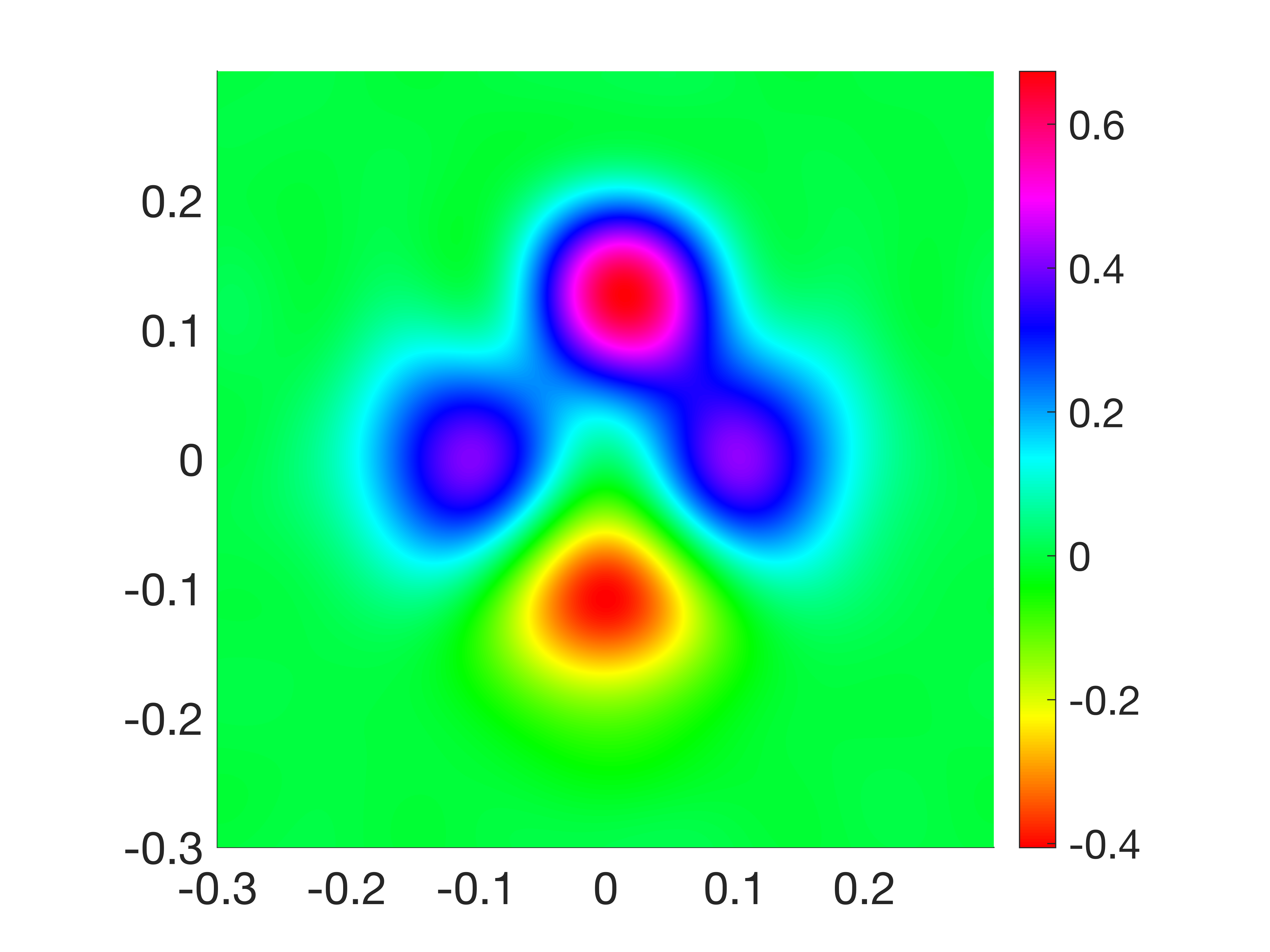}
	\caption{The recovered Fourier mode $\widetilde{f}_n$ of the source $f_1$. Row 1: surface plots; Row 2: contour plots. Column 1: $n=1$; Column 2: $n=15$; Column 3: $n=30$.}\label{fig:S1}
\end{figure}

Next, we reconstruct the source at different slices: $x_3=0.15, 0.85$, and $1.35$. The reconstructions are depicted in \Cref{fig: S1_2}, which illustrate that the profile of the source function can be well captured at these locations. 

This example shows that when the exact source has limited Fourier modes, these Fourier modes can be well-reconstructed. Furthermore, the source can also be recovered satisfactorily under the Fourier expansion. However, most of the source functions may not have limited Fourier expansions. So we shall reconstruct a more general source function in the next example.

\begin{figure}[!h]
	\centering
	\includegraphics[width=0.3\textwidth]{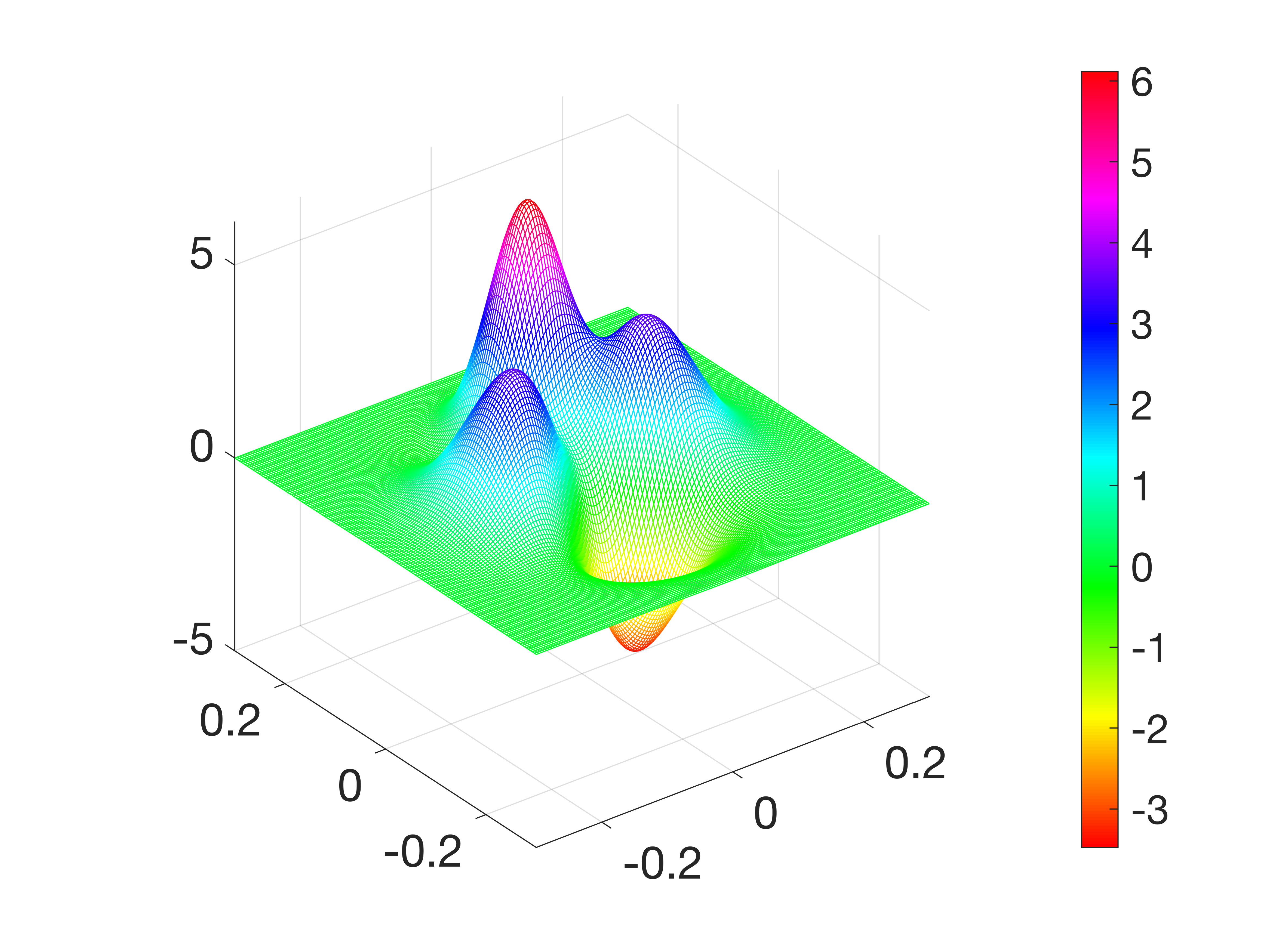}\quad
	\includegraphics[width=0.3\textwidth]{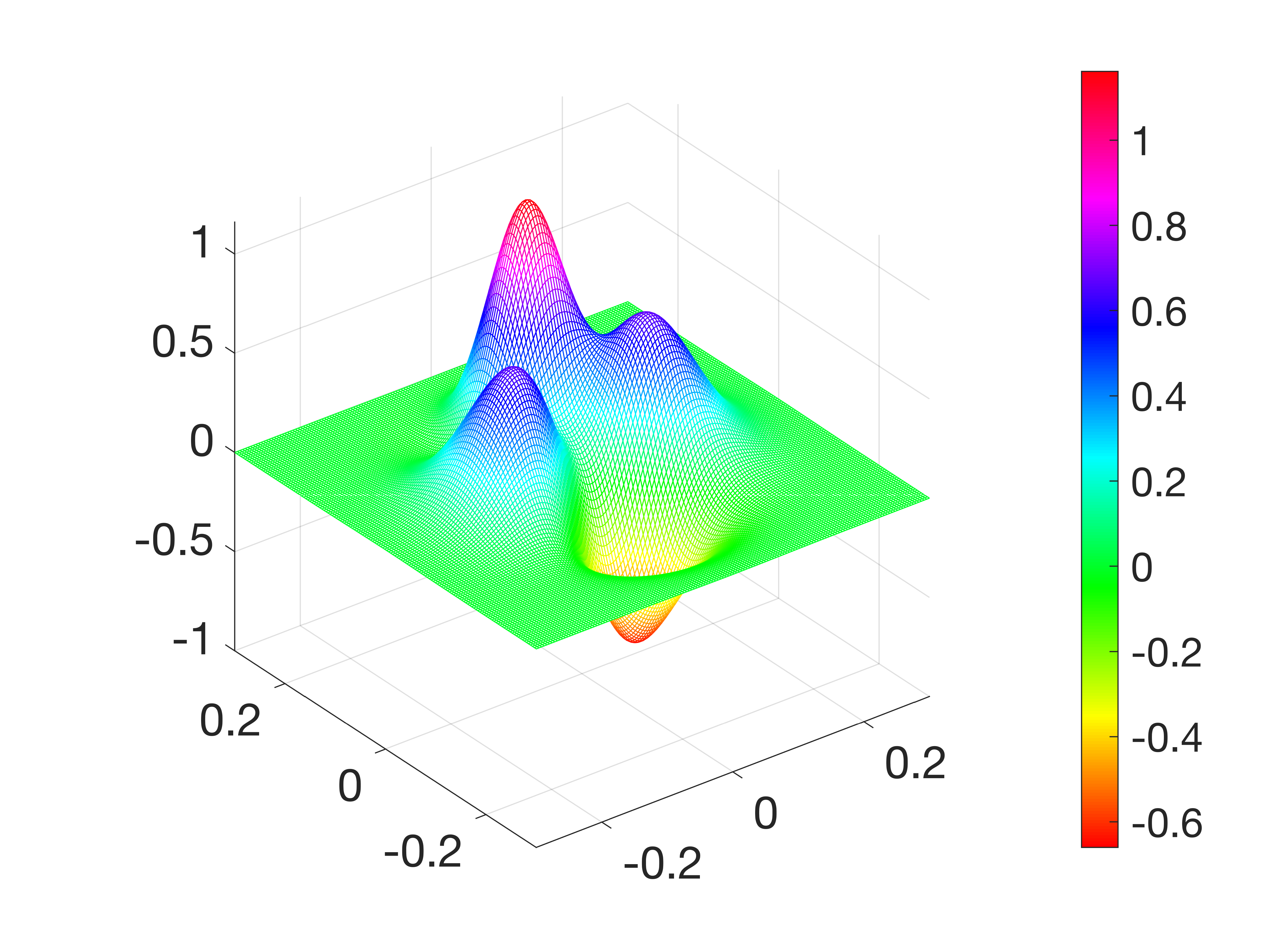}\quad
	\includegraphics[width=0.3\textwidth]{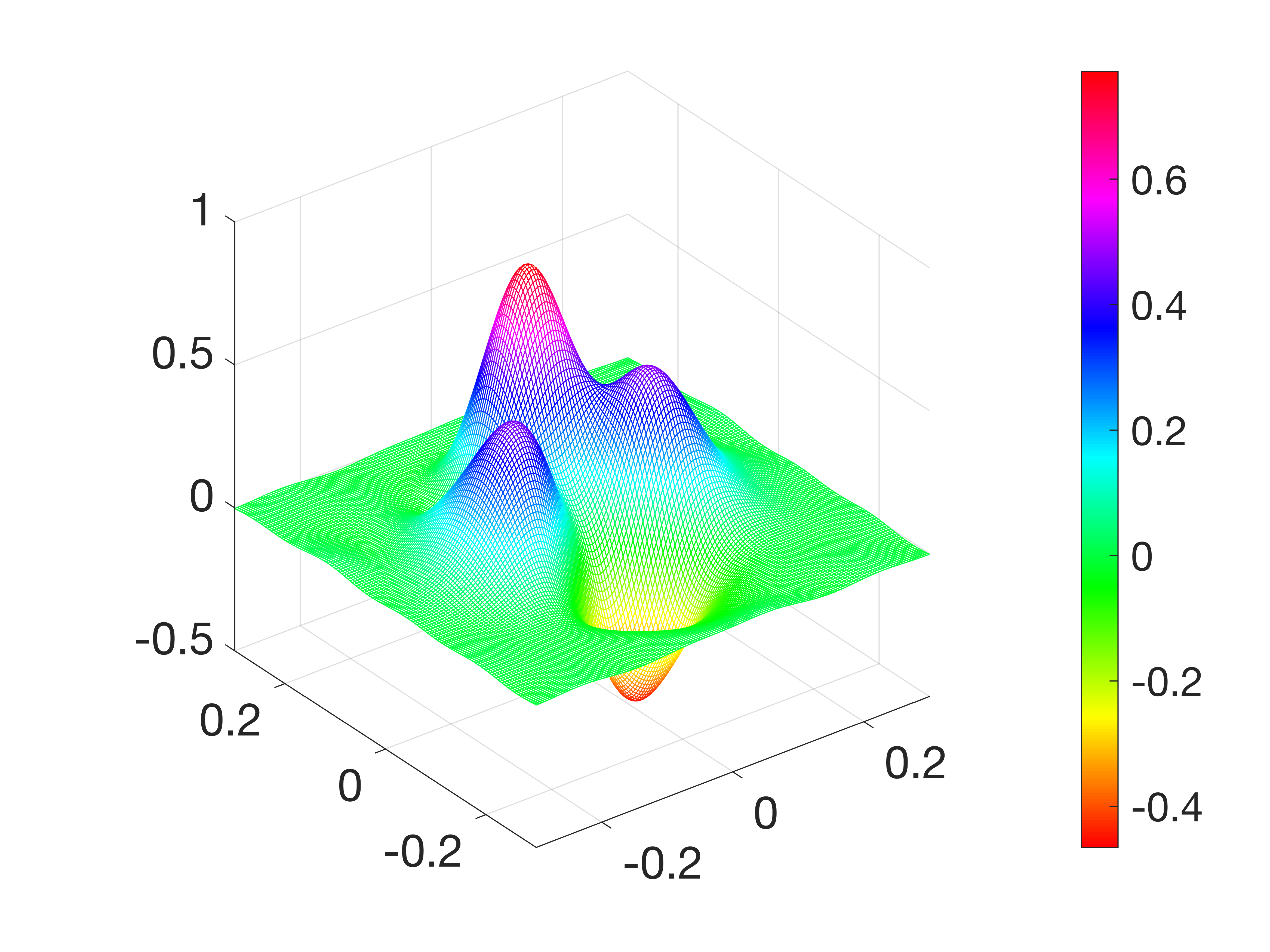}\\
	\includegraphics[width=0.3\textwidth]{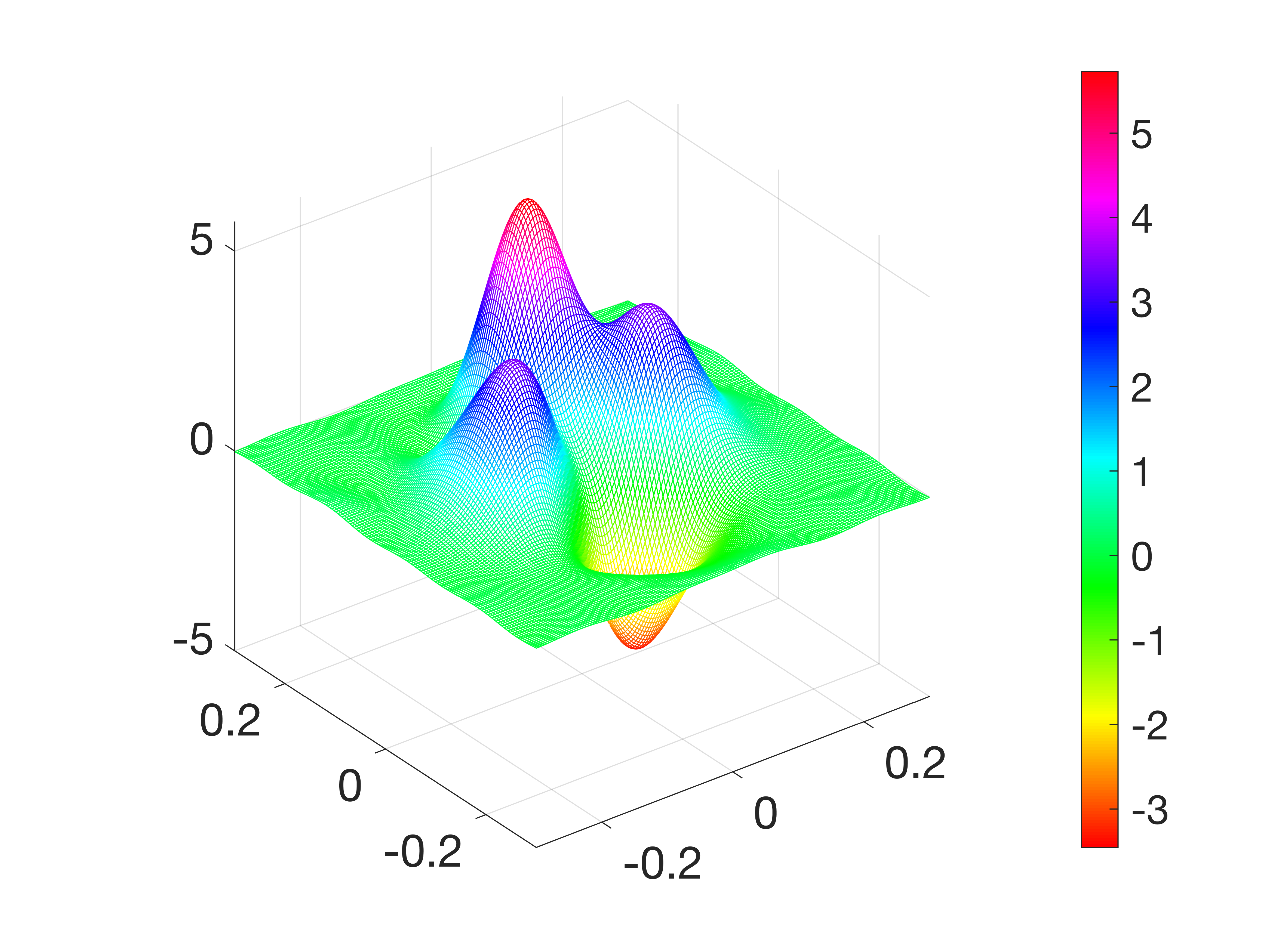}\quad
	\includegraphics[width=0.3\textwidth]{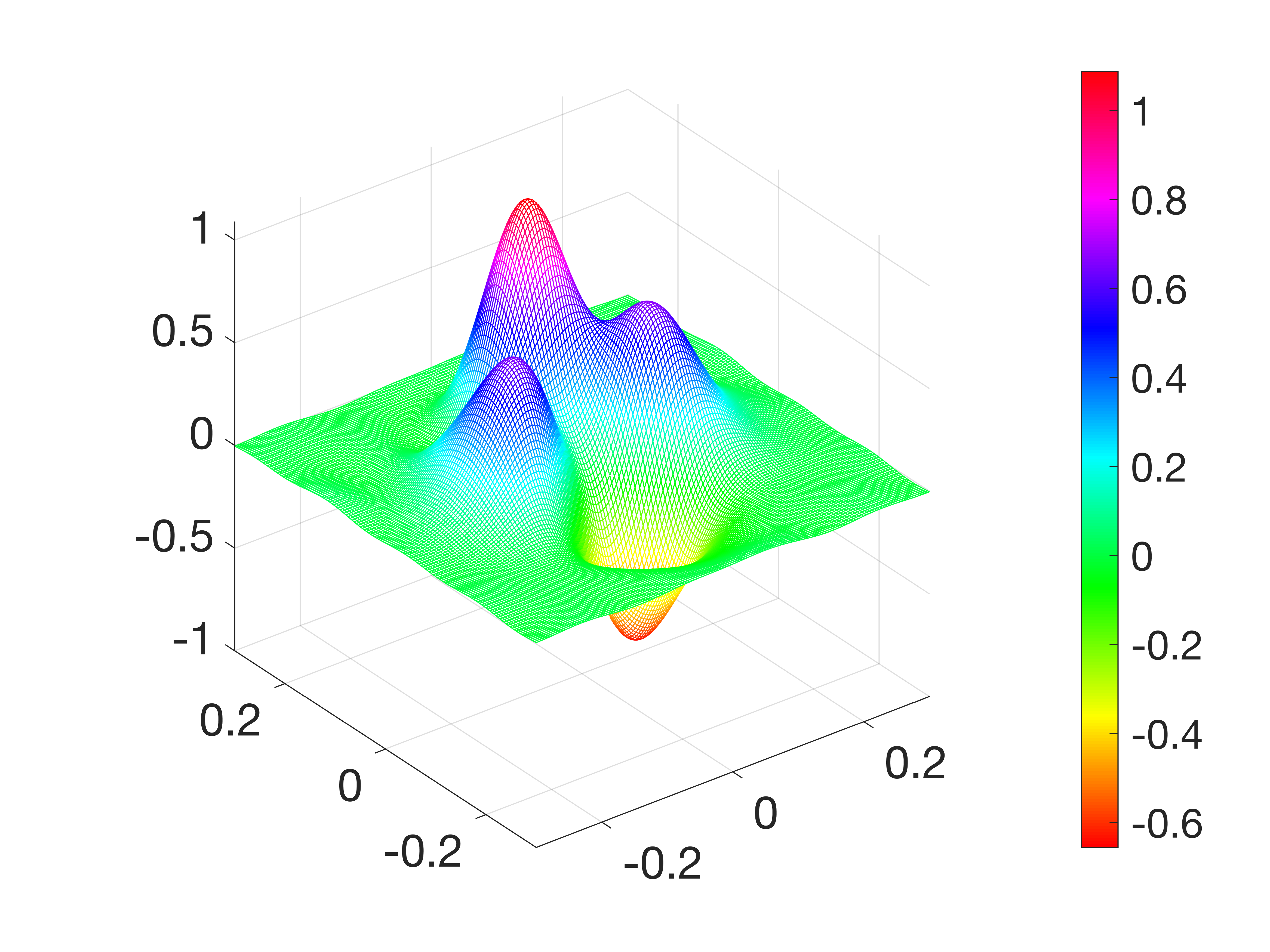}\quad
	\includegraphics[width=0.3\textwidth]{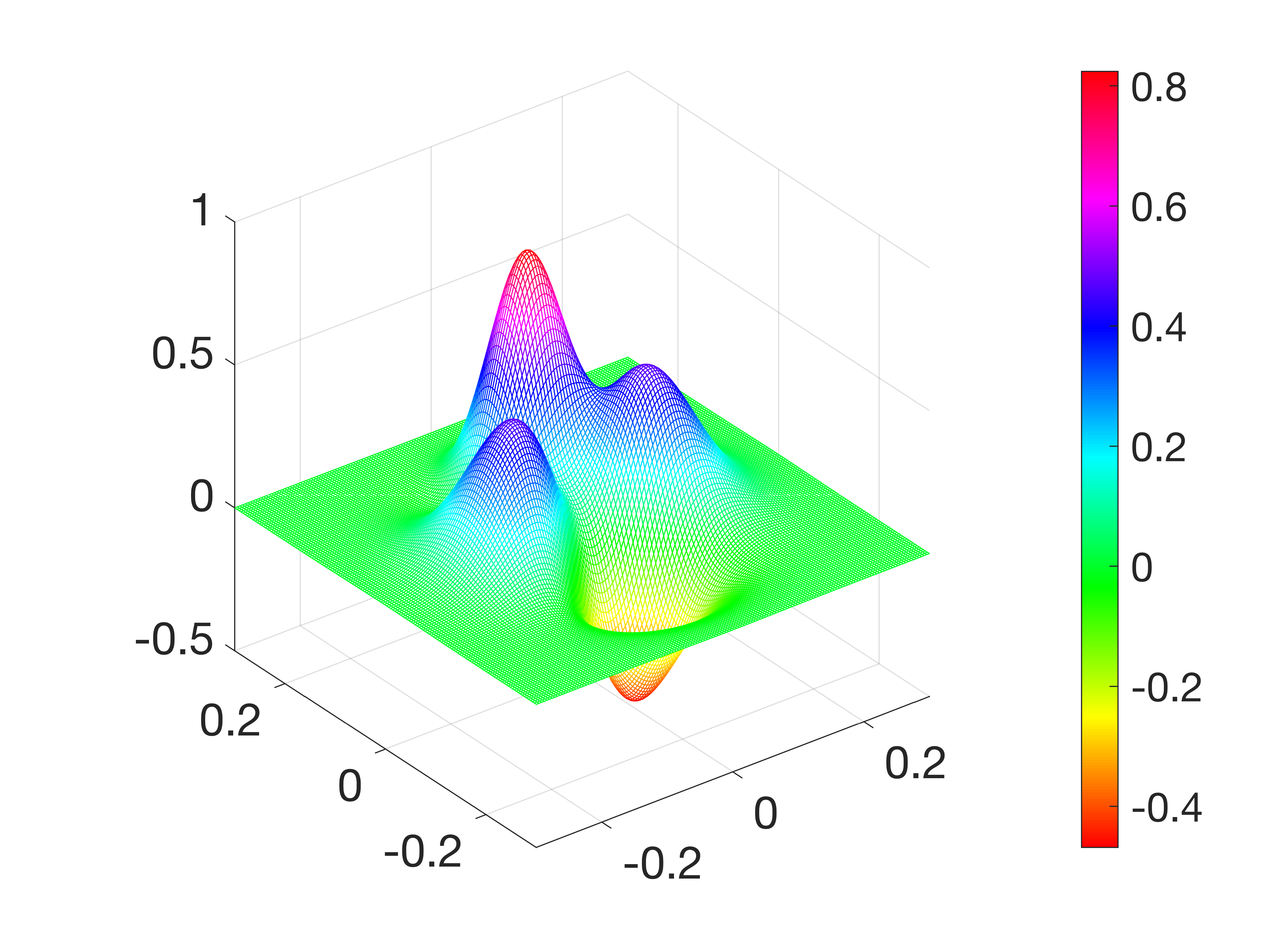}
	\caption{The exact source $f_1$ and recovered source $f_N$ at different $x_3$. Row 1: exact sources; Row 2: reconstructions; Column 1: $x_3=0.15$; Column 2: $x_3=0.85$; Column 3: $x_3=1.35$.}\label{fig: S1_2}
\end{figure}

\begin{example}
	The second example is devoted to reconstructing a mountain-shaped function described by 
	\[
	f_2(x_1,x_2,x_3)=1.1\mathrm{e}^{-200\left((x_1-0.01)^2+(x_2-0.12)^2+x_3^2\right)}
	-100\left(x_1^2-x_2^2\right)\mathrm{e}^{-90\left(x_1^2+x_2^2+x_3^2\right)}.
	\]
\end{example}

Different from $f_1$ which has limited Fourier expansion with the same Fourier mode, the source function $f_2$ is more general.  
Here the truncation $N=N(\delta)$ is set to be $5\left[\log\delta\right],$ with $[X]$ denoting the largest integer that is smaller than $X+1$. In this way, the influence of the magnitude of noise is investigated in this example. For a quantitative evaluation of the inversion scheme, we compute the relative $L^2$ errors 
$$
Err=\frac{\|f_N-f\|_{L^2(V_1)}}{\|f\|_{L^2(V_1)}}\ \text{and}\ err(x_3)=\frac{\|f_N(\cdot, x_3)-f(\cdot, x_3)\|_{L^2(V_0)}}{\|f(\cdot, x_3)\|_{L^2(V_0)}},\quad x_3\in[0,L].
$$

First, $\delta=5\%$ is used and we compare $f$ and $f_N$ at  $x_3=0.25$ in \Cref{fig:S2_1}. One can see from the results that the source function is well-reconstructed. For an in-depth slice view, we plot several cross-sections of them in \Cref{fig: S2_2}. We can see from \Cref{fig: S2_2} that the reconstruction matches almost perfectly with the exact source at these typical cross sections. Especially, even if the source value is close to $0,$ the reconstruction is still satisfactory (see \Cref{fig: S2_2}(b) for example).

\begin{figure}
	\centering
	\includegraphics[width=0.4\textwidth]{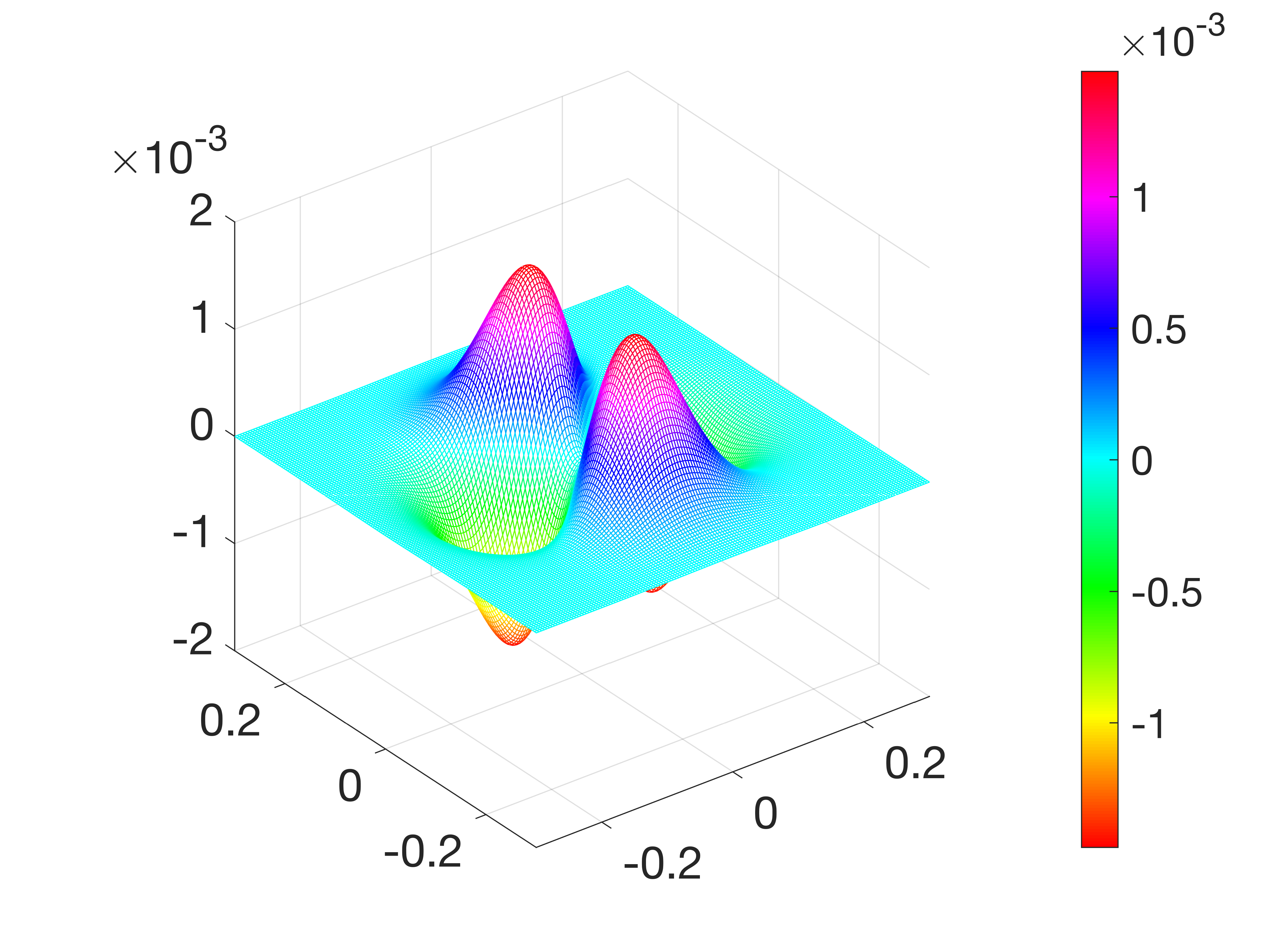}\quad
	\includegraphics[width=0.4\textwidth]{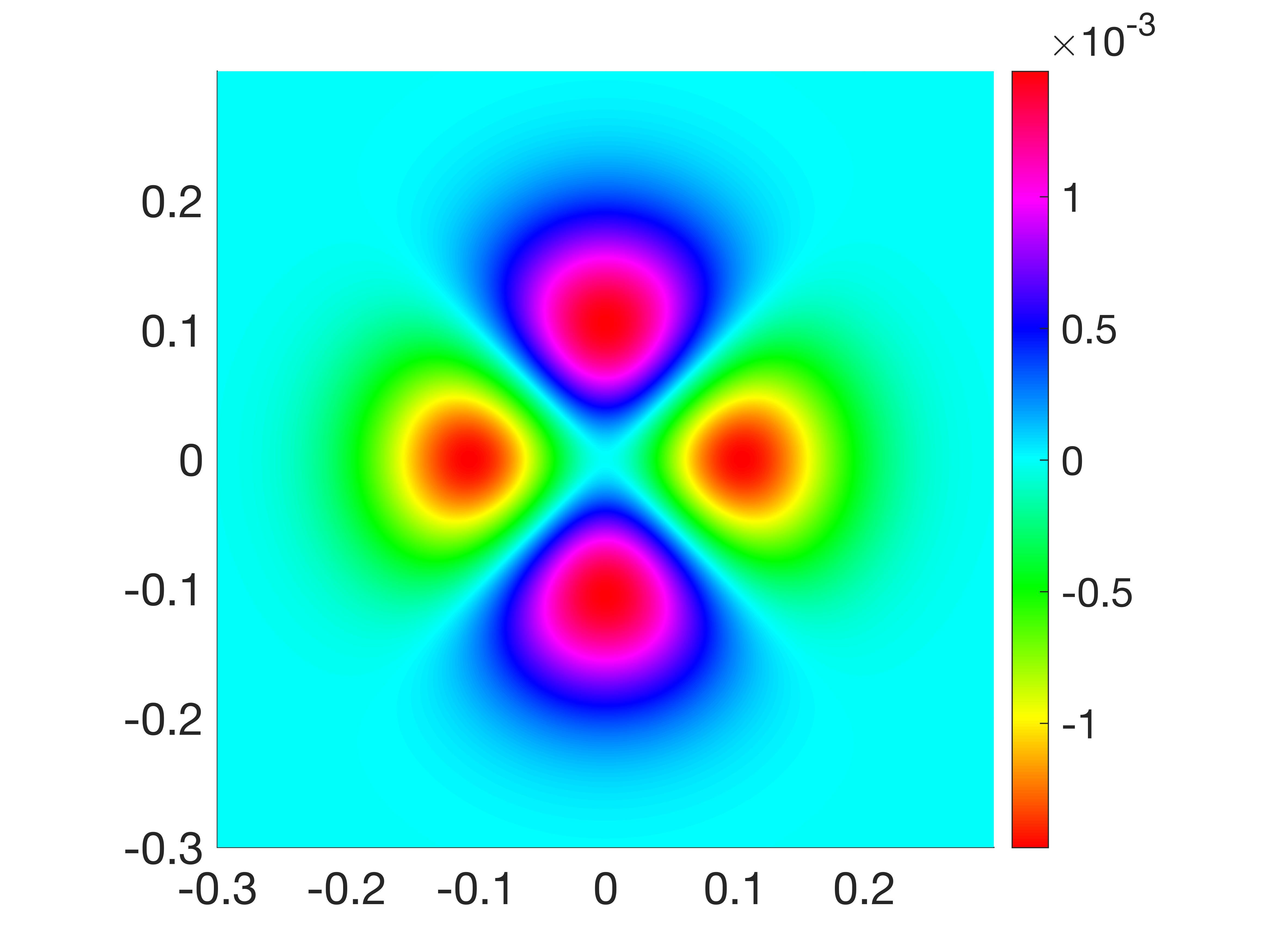}\\
	\includegraphics[width=0.4\textwidth]{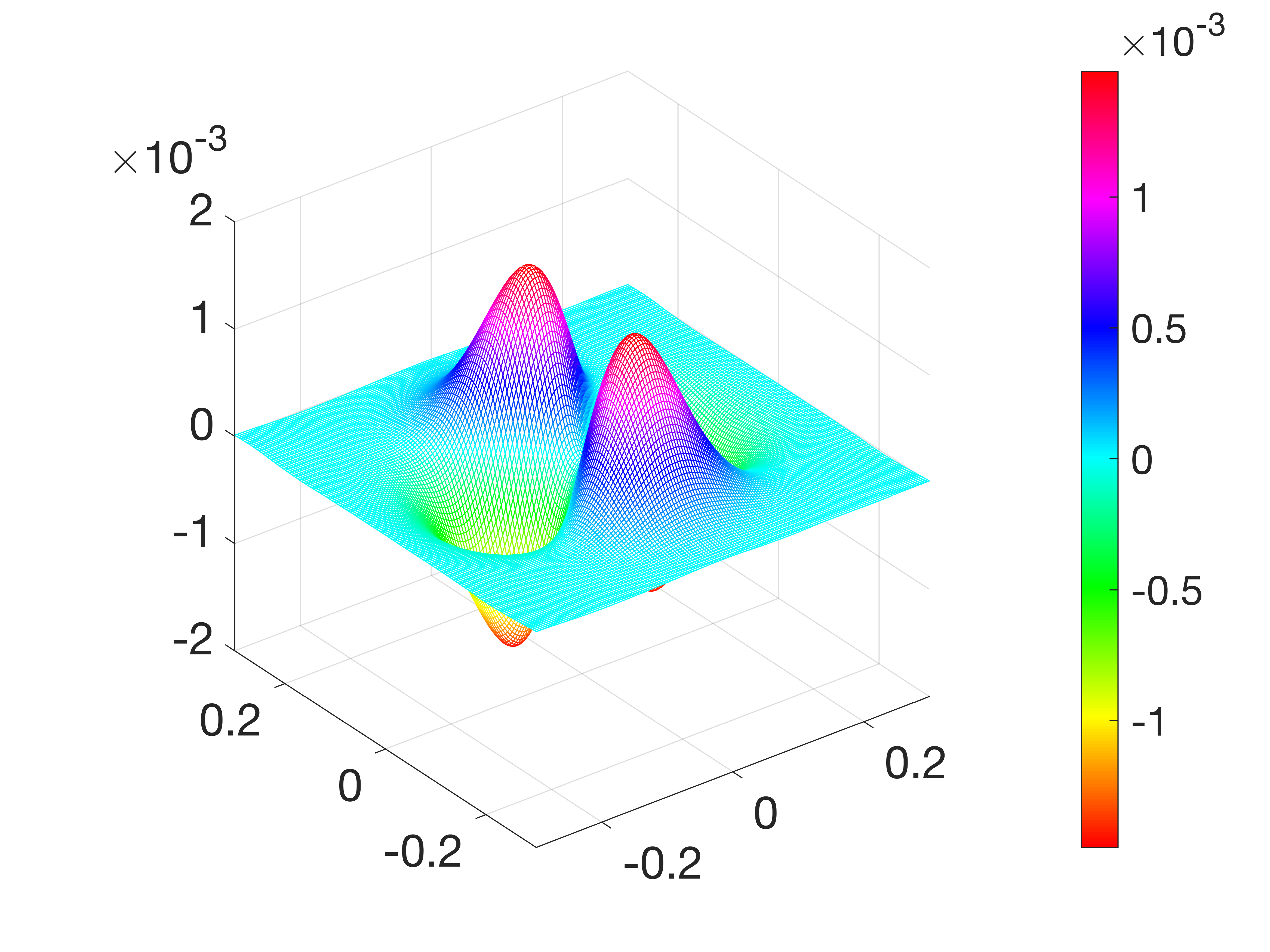}\quad
	\includegraphics[width=0.4\textwidth]{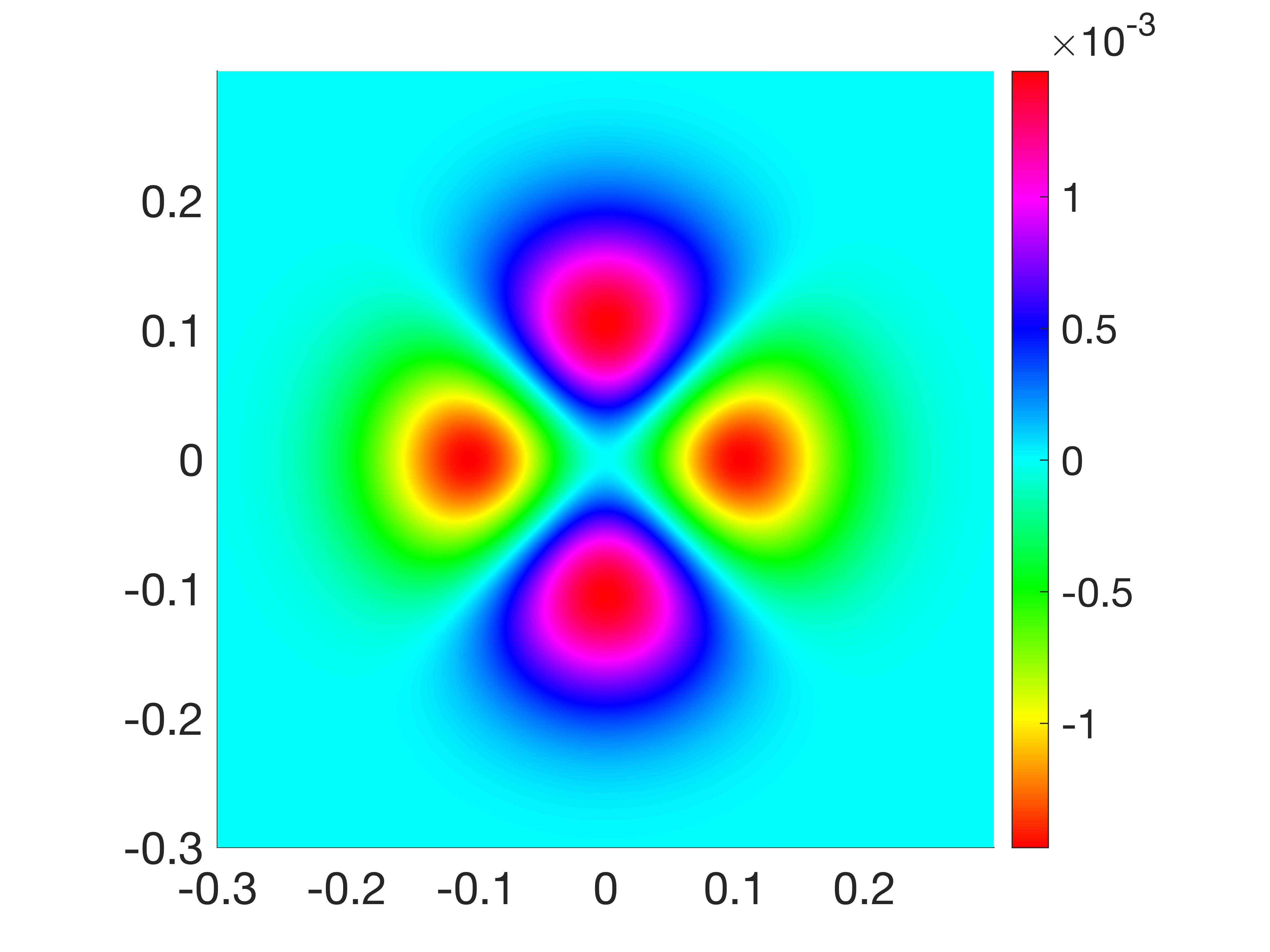}\\
	\caption{The exact source $f_2$ and recovered source $f_N$ at $x_3=0.25$. Row 1: exact source. Row 2: reconstruction. Column 1: surface plots; Column 2: contour plots.}\label{fig:S2_1}
\end{figure}

\begin{figure}
	\centering
	\subfloat[$x_2=-0.002,x_3=0.4$]{\includegraphics[width=0.4\textwidth]{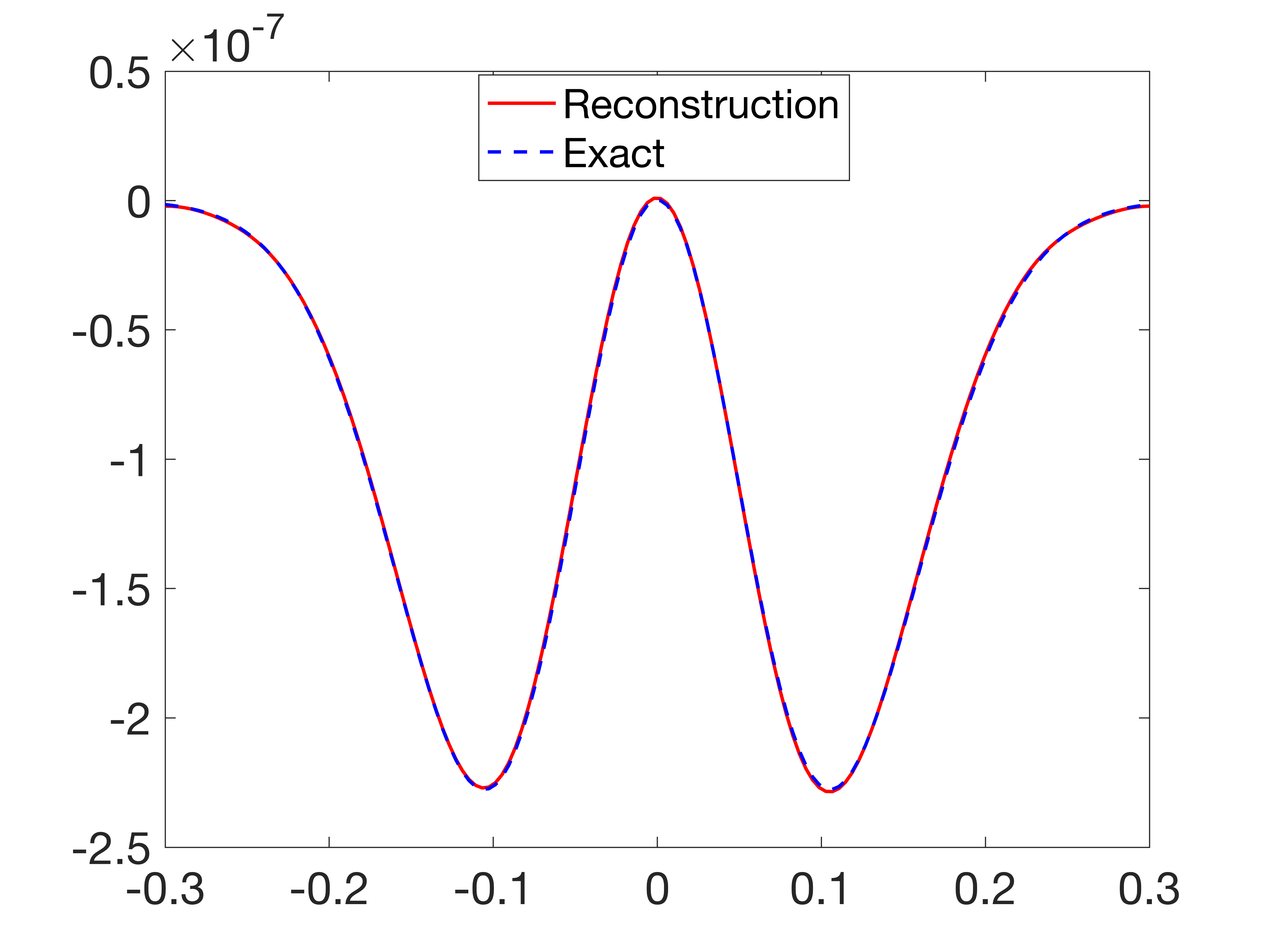}}\quad
	\subfloat[$x_2=-0.002,x_3=0.6$]{\includegraphics[width=0.4\textwidth]{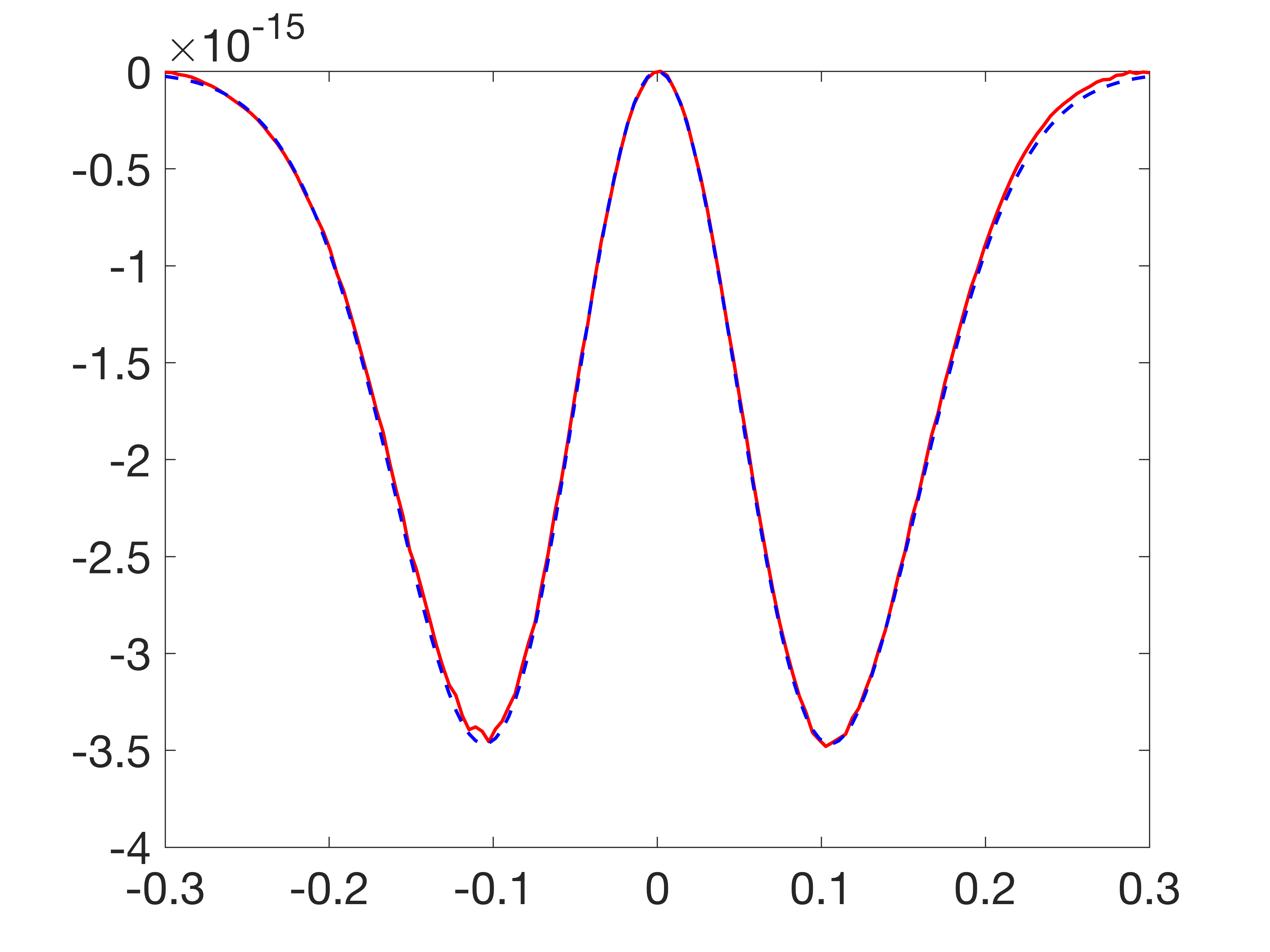}}\\
	\subfloat[$x_2=-0.002,x_3=0.15$]{\includegraphics[width=0.4\textwidth]{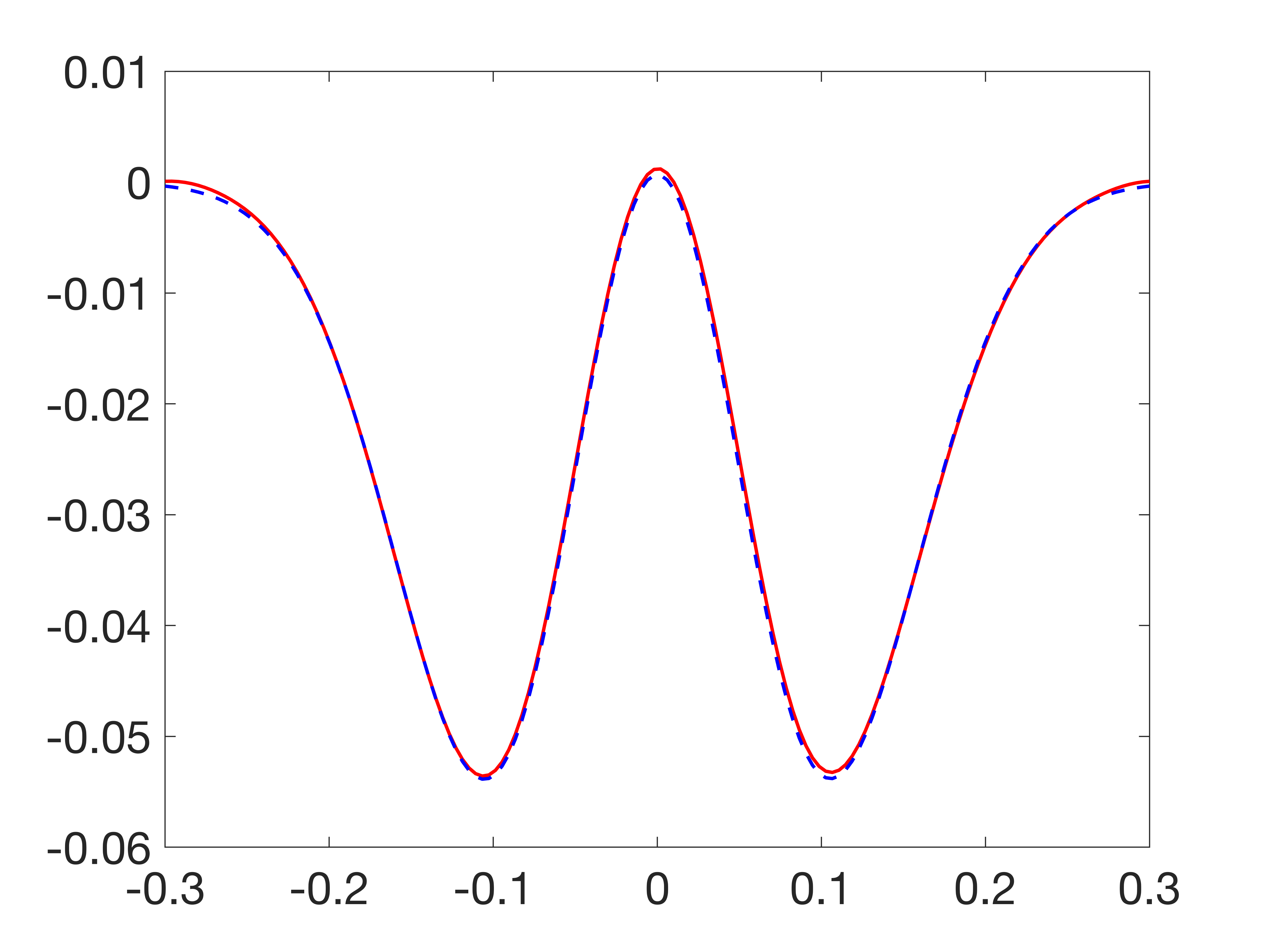}}\quad
	\subfloat[$x_2=-0.082,x_3=0.15$]{\includegraphics[width=0.4\textwidth]{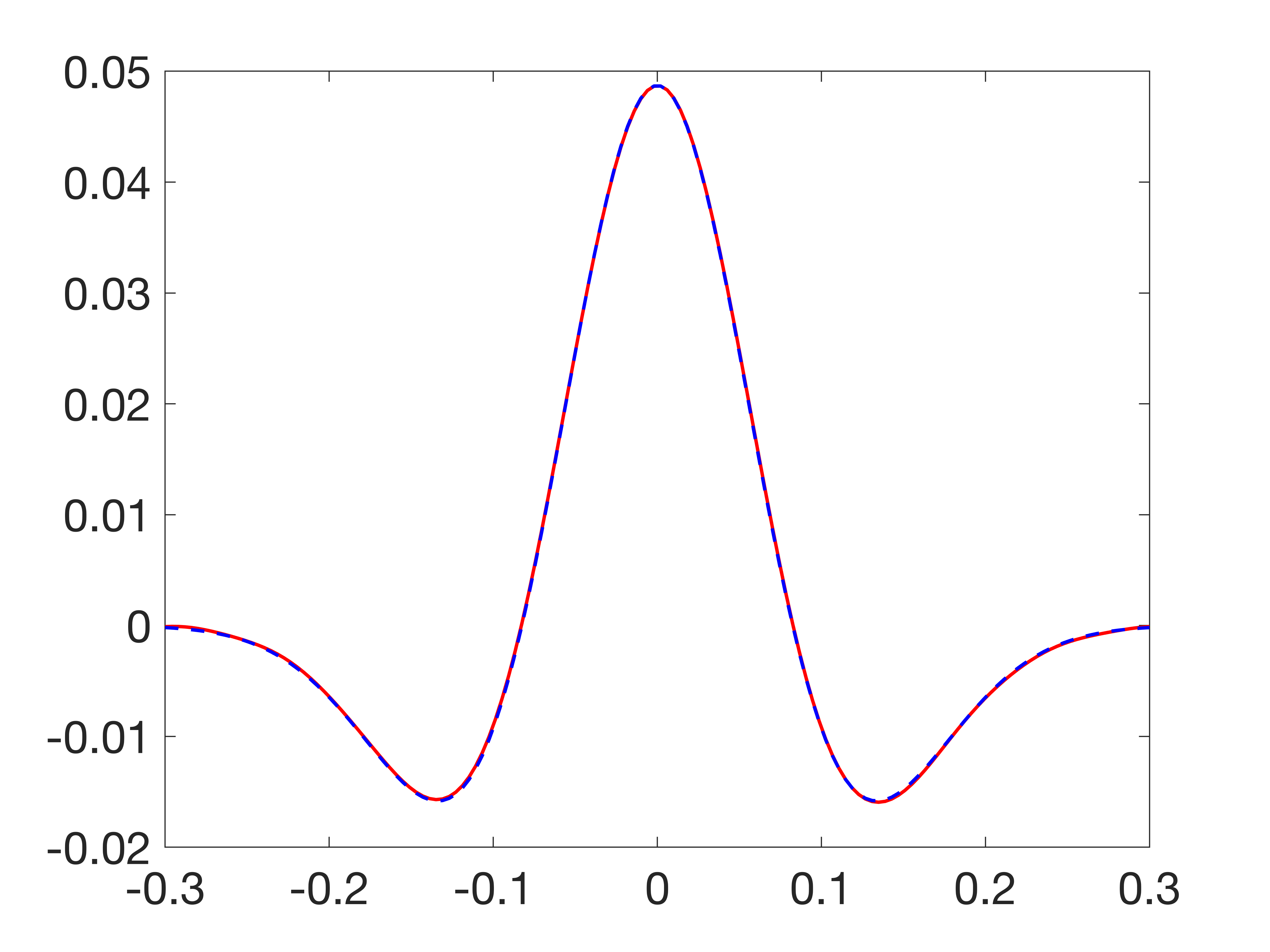}}\\
	\caption{The reconstructed source is plotted against the exact source $f_2$ at different cross-sections.}\label{fig: S2_2}
\end{figure}

Further, we compute the recovery errors at different $x_3$ locations subject to different noise levels $\delta.$ \Cref{tab: error} shows that the reconstructions are reasonably stable in the sense that the quality of the inversion improves as the noise level decreases. 

\begin{table}[htbp]
	\centering
	\caption{The relative errors of the reconstruction of $f_2$ with different noise levels $\delta.$}\label{tab: error}
	\begin{tabular}{ccccc}
			\toprule  
			$\delta$ & $1\%$ & $5\%$ & $10\%$ & $20\%$
			\\ 
			\midrule  
			$N(\delta)$ & $20$        & $10$        & $10$         & $5$\\
			$err(0.2)$     & $1.18\%$ & $1.96\%$ & $3.72\%$ & $5.27\%$\\
			$err(0.3)$     & $1.15\%$ & $1.90\%$ & $3.81\%$ & $8.42\%$\\
			$err(0.4)$     & $1.23\%$ & $1.52\%$ & $3.20\%$ & $6.30\%$\\
			$Err$             & $2.77\%$ & $3.03\%$ & $4.65\%$ & $9.39\%$\\
			\bottomrule 
		\end{tabular}
\end{table}

\subsection{The far-field case}

In this last subsection, we briefly discuss the inverse source problem using the far-field data. We assume $V = 0$ in this case.
Under Assumption (A) the outgoing solution to \eqref{model} can be represented by
\[
u(x, k) =  \frac{\rm i}{4}  \sum_{n = 1}^{N_0} \int_{\widetilde{B}_R} H_0^{(1)} (\beta_n(k) |\tilde{x} - \tilde{y}|) f_n(\tilde{y}){\rm d}\tilde{y}  \sin(\alpha_n x_3).
\]
From the asymptotic expansion of the fundamental solution $\mathrm{i}H^{(1)}_0(  k| \tilde{x} -   \tilde{y}|)/4$ of the two-dimensional Helmholtz equation \cite{CK} as $|\tilde{x}|\to\infty$, we have 
\begin{align*}
	u(x, k) = \frac{\rm i}{4}  \sum_{n=1}^{N_0}\frac{1 + \rm i}{4\sqrt{\pi k}} \frac{\mathrm{e}^{{\rm i}  \beta_n(k) |\tilde{x}|}}{\sqrt{|\tilde{x}|}} \int_{\mathbb R^2} \mathrm{e}^{-{\rm i}  
		\beta_n(k)  \tilde{y}\cdot \hat{\tilde{x}}} f_n(\tilde{y}) {\rm d}  \tilde{y} \sin(\alpha_n x_3)
	+  \mathcal{O}\left(\frac{1}{| \tilde{x}|}\right), 
\end{align*}
where $\hat{\tilde{x}} = \tilde{x}/|\tilde{x}|\in\mathbb S=\{\tilde{x}\in\mathbb{R}^2:\, |\tilde{x}|=1\}$ is the observation angle. Since $\mathrm{e}^{{\rm i}  \beta_n(k) |\tilde{x}|}$ is exponentially decaying for those $n$ such that $k<\alpha_n$, we define the far-field pattern $\mathcal{A}_\infty(k, \hat{\tilde{x}}, x_3)$ as follows
\begin{align}\label{far}
	\mathcal{A}_\infty(  k, \hat{\tilde{x}}, x_3) &=  \sum_{\{n: k>\alpha_n\}}\int_{\mathbb R^2} \mathrm{e}^{-{\rm i}  
		\beta_n(k)  \tilde{y}\cdot \hat{\tilde{x}}} f_n(\tilde{y}) {\rm d}  \tilde{y}  \sin(\alpha_n x_3)\notag \\
	&= \sum_{\{n: k>\alpha_n\}} \hat{f}_n(\beta_n(k) \hat{\tilde{x}})  \sin(\alpha_n x_3).
\end{align}

Let $I_1$ be an interval such that $I_1\subset (\alpha_1, \alpha_2)$  and let $x_3$ be fixed such that $\sin\alpha_1 x_3\neq 0$.
As can be seen in \eqref{far}, if the multi-wavenumber far-field data $\{\mathcal{A}_\infty(  k, \hat{\tilde{x}}, x_3):  \hat{\tilde{x}}\in\mathbb S, k\in I_1\}$ is given,
then
\[
\mathcal{A}_\infty\left( k, \hat{\tilde{x}}, x_3\right) = \hat{f}_1\left(\beta_2(k) \hat{\tilde{x}}\right)  \sin(\alpha_1 x_3).
\]
This implies that the far field of only one propagating mode corresponding to $f_1$ is detected which gives the following Fourier transform
\[
\left\{\hat{f}_1(\tilde{\xi}): |\tilde{\xi}|< \sqrt{\alpha_2^2  - \alpha_1^2}\right\}.
\]
Thus, by the analyticity of $\hat{f}_1(\tilde{\xi})$ and the inverse Fourier transform we can determine $f_1$ and even derive a stability estimate
by analytic continuation. 
Moreover, let $I_2\subset (\alpha_2, \alpha_3)$, if $f_1$ is recovered and 
$\{\mathcal{A}_\infty(  k, \hat{\tilde{x}}, x_3):  \hat{\tilde{x}}\in\mathbb S, k\in I_2\}$ is available, in a similar manner
we can determine $f_2$ by detecting the far-field of the corresponding single mode. Proceeding recursively in this way, we get the following uniqueness result.

\begin{theorem}\label{uisp_1}
	Let $x_3$ be chosen such that $\sin\alpha_n x_3 \neq 0$ for $1\leq n\leq  N_0$. The far-field $\{\mathcal{A}_\infty(  k, \hat{\tilde{x}}, x_3):  \hat{\tilde{x}}\in\mathbb S, k\in \cup_{i=1}^{N_0} I_i, \, I_i\subset (\alpha_i, \alpha_{i+1})\}$ uniquely determines $f$.
\end{theorem}

\begin{remark}
	Compared with Theorem \ref{uisp}, Theorem \ref{uisp_1} requires data measured at more wavenumbers. Physically speaking, a possible reason accounting for the necessity of this extra data is that part of the information is only involved in the evanescent modes which decay drastically. Hence, this information cannot be accessed or retrieved from the far-field data. Thus, more far-field data is inevitably needed to compensate for the lack of information towards establishing the uniqueness.
\end{remark}

\begin{remark}
	We would like to highlight that the uniqueness of the inverse source problem is established by incorporating the far-field due to a single mode one by one from low to high wavenumbers. This framework of derivation has the advantage of avoiding the situation where multiple propagating modes are concurrently present. Conversely, for instance, if we only collect the data
	$\{\mathcal{A}_\infty(  k, \hat{\tilde{x}}):  \hat{\tilde{x}}\in\mathbb S, k\in I_2\}$, then we are only able to find the superposed quantity $\hat{f}_1(\beta_1(k) \hat{\tilde{x}}) + \hat{f}_2(\beta_2(k) \hat{\tilde{x}})$. In this case, it would be difficult to separate and recover either $f_1$ or $f_2$ from the sum.
\end{remark}

\section{Inverse problem II: determining the potential}\label{sec: IP2}

In this section, an inverse potential problem in the waveguide is considered. We assume that $V$ is real-valued and $f\equiv 0$ in this section.
The key ingredient in the analysis is applying results in \Cref{main_d} and an argument of analytic continuation.

Let $k>\alpha_1$ and $d\in\mathbb S$ be respectively the wavenumber and incident direction, and denote $u^{\rm inc}(x, k, d) = \mathrm{e}^{{\rm i}\sqrt{k^2 - \alpha_1^2}\tilde{x}\cdot d}\sin\alpha_1 x_3$ the incident field. Then the total field is given by $u = u^{\rm inc}  + u^s$ where $u^s$ is the scattered field produced by $u^{\rm inc}$ and the potential $V(\tilde{x})$.
Consider the following homogeneous  Schr\"odinger equation with $u$ satisfying the boundary conditions \eqref{bc}
\begin{equation}\label{model_1}
	-\Delta u+ V u - k^2 u =0,\quad \text{in } D.
\end{equation}

We are interested in the inverse problem of determining $V$ from $u(x, k, d)$ on $\Gamma_R$. Here we employ $u(x, k, d)$ to signify the dependence of 
$u$ on $k$ and $d$.

Under the above configuration, the scattered field satisfies 
\begin{equation}\label{eqn_1}
		-\Delta u^s+ V u^s - k^2 u^s =-Vu^{\rm inc} ,\quad \text{in } D
\end{equation}
and the boundary conditions \eqref{bc}. Multiplying both sides of \eqref{eqn_1} by the factor $\mathrm{e}^{{\rm i}\sqrt{k^2 - \alpha_1^2}\tilde{x}\cdot d_1}\sin\alpha_1 x_3$ with $d_1\in\mathbb S$ and integrating by parts over $C_R$, we obtain
\begin{equation}\label{eqn_2}
	\begin{split}
		&\quad\int_{\widetilde{B}_R} V \mathrm{e}^{{\rm i}\sqrt{k^2 - \alpha_1^2}(d + d_1)}{\rm d}\tilde{x}\\ 
		&= \int_{\partial\widetilde{B}_R}\!\! \left(\partial_{\nu_{\tilde{x}}}u_1(\tilde{x}, k) \mathrm{e}^{{\rm i}\sqrt{k^2 - \alpha_1^2}\tilde{x}\cdot d_1}
		- {\rm i}\sqrt{k^2 - \alpha_1^2} d_1\cdot \nu_{\tilde{x}} u_1(\tilde{x}, k) \mathrm{e}^{{\rm i}\sqrt{k^2 - \alpha_1^2}\tilde{x}\cdot d_1}\!\right)\! {\rm d}\tilde{x}\\
		&\quad- \int_{C_R} V u^s \mathrm{e}^{{\rm i}\sqrt{k^2 - \alpha_1^2}\tilde{x}\cdot d_1}\sin\alpha_1 x_3{\rm d}x\\
		&=  \int_{\partial\widetilde{B}_R} \!\! \left(\partial_{\nu_{\tilde{x}}}u_1(\tilde{x}, k) \mathrm{e}^{{\rm i}\sqrt{k^2 - \alpha_1^2}\tilde{x}\cdot d_1}
		- {\rm i}\sqrt{k^2 - \alpha_1^2} d_1\cdot \nu_{\tilde{x}} u_1(\tilde{x}, k) \mathrm{e}^{{\rm i}\sqrt{k^2 - \alpha_1^2}\tilde{x}\cdot d_1} \!\right)\!{\rm d}\tilde{x}\\
		&\quad+\mathcal{O}\left(\frac{1}{k}\right).
	\end{split}
\end{equation}
where $u_1$ is the first Fourier mode of $u$. As the inhomogeneous term $-Vu^{\rm inc}$ on the right-hand side of the equation \eqref{eqn_1} has a single Fourier mode in the $x_3$ variable,
in the last equality, we can use the resolvent estimate in Theorem \ref{main_d}  to derive
\[
\int_{C_R} V u^s \mathrm{e}^{{\rm i}\sqrt{k^2 - \alpha_1^2}\tilde{x}\cdot d_1}\sin\alpha_1 x_3{\rm d}x = \mathcal{O}\Big(\frac{1}{k}\Big).
\]
Notice that $\left\{\sqrt{k^2 - \alpha_1^2}(d + d_1): d, d_1\in\mathbb S\right\} = \left\{\xi: |\xi|\leq 2\sqrt{k^2 - \alpha_1^2}\right\}.$ 
We next show that the boundary measurements $\{u(x, k, d): x\in\Gamma_R, k\in I, d\in\mathbb S\}$ uniquely determine $V$. Here $I = [K_0, K_1]$ with
$K_1>K_0>\max\{M, \alpha_1\}$ where $M$ is specified in Theorem \ref{meromorphic}.

Let $V_1$ and $V_2$ be two potential functions.
Assume that $u^{(1)}$ and $u^{(2)}$ are solutions to \eqref{model_1} corresponding to $V_1$ and $V_2$, respectively. Denote $W = V_1 - V_2$ and
$v = u^{(1)} - u^{(2)}$. Substituting $u$ and $V$ in \eqref{eqn_2} by $u^{(1)}, V_1$ and $u^{(2)}, V_2$, respectively, and taking subtraction yields
\begin{align*}
		&\int_{\widetilde{B}_R} W \mathrm{e}^{{\rm i}\sqrt{k^2 - \alpha_1^2}(d + d_1)}{\rm d}\tilde{x}\\ 
		&=  \int_{\partial\widetilde{B}_R}\left(\partial_{\nu_{\tilde{x}}}v_1(\tilde{x}, k) \mathrm{e}^{{\rm i}\sqrt{k^2 - \alpha_1^2}\tilde{x}\cdot d_1}
		- {\rm i}\sqrt{k^2 - \alpha_1^2} d_1\cdot \nu_{\tilde{x}} v_1(\tilde{x}, k) \mathrm{e}^{{\rm i}\sqrt{k^2 - \alpha_1^2}\tilde{x}\cdot d_1}\right) {\rm d}\tilde{x}\\
		&\quad+\mathcal{O}\Big(\frac{1}{k}\Big).
\end{align*}

Now it suffices to show $W = 0$ if $v(x, k, d) = 0$ for all $d\in\mathbb S$ and $k\in I$. As $v = v(x, k, d)$ is analytic for $k\in I$, we immediately have 
$v(x, k, d) = 0$ for all $k\geq M$. Moreover, from the Dirichlet-to-Neumann map the normal derivative $\partial_{\nu_{\tilde{x}}} v_1$ can be computed from
\begin{equation*}
\partial_{\nu_{\tilde{x}}} v_1 = Tv_1 = \sum_{m\in\mathbb Z} \beta_1 a_m^1 \frac{{H_m^{(1)}}^\prime(\beta_1 R)}{H_m^{(1)}(\beta_1 R)} \mathrm{e}^{{\rm i}m\theta}\sin\alpha_1 x_3.
\end{equation*}
where $a_m^1$  are the Fourier coefficients
\[
a_m^1 = \sqrt{\frac{\pi}{LR}}\int_{C_R} \mathrm{e}^{-{\rm i}m\theta} \sin(\alpha_1x_3) \, v {\rm d}x, \quad x = (r\cos\theta, r\sin\theta, x_3).
\]
Thus, $v(x, k, d) = 0$ gives $\partial_{\nu_{\tilde{x}}} v(x, k, d) = 0$. Consequently, from \eqref{eqn_2} we have
\begin{equation}\label{eqn_3}
	\widehat{W}(\xi) \leq \mathcal{O}\Big(\frac{1}{k}\Big) \quad \text{for all} \, |\xi|\leq 2\sqrt{k^2 - \alpha_1^2}.
\end{equation}
As $v(x, k, d) = 0$ on $\Gamma_R$ holds for all $k\geq M$, by letting $k\to\infty$ in \eqref{eqn_3} we arrive at 
\[
\widehat{W}(\xi) = 0,\quad \xi\in\mathbb R^2,
\]
which yields $W = 0$ by the inverse Fourier transform. This implies the uniqueness of the inverse problem. In summary, we have the following theorem.

\begin{theorem}
	The multi-frequency measurement $\{u(x, k, d): x\in\Gamma_R, k\in I, d\in\mathbb S\}$ uniquely determines $V$.
\end{theorem}

Now we discuss the stability.
Assume that $u_1^s(x,  k,  d)$ and $u_2^s(x,  k,  d)$ are the scattered field corresponding to the incident wave $u^{\rm inc} = \mathrm{e}^{{\rm i}\sqrt{k^2 - \alpha_1^2}\tilde{x}\cdot d}\sin\alpha_1 x_3$ and potentials $V_1$ and $V_2$, respectively.
Let $s>0$ be an arbitrary positive constant. Define a real-valued function space
\begin{align*}
\mathcal E_Q = \{V \in H^{s}(D)\cap L^\infty(D):\ & \|V\|_{H^{s}(D)}\leq Q,\ \|V\|_{L^{\infty}(D)}\leq Q,\\
& \mathrm{supp}V\subset C_R, ~ V: D \rightarrow \mathbb R \}.
\end{align*}

Utilizing the resolvent estimates and the estimate \eqref{eqn_2}, we have the following stability estimate (the proof follows \cite{ZZ} in a straightforward way by applying the quantitative analytic continuation. Thus we omit it for brevity).

\begin{theorem}\label{main_2d}
	Let $V_1, V_2\in \mathcal E_Q$. The following increasing stability estimate holds
	\begin{align}\label{stability_1}
		\|V_1 - V_2\|_{L^2(D)}^2
		\lesssim K^\alpha\epsilon^2+\frac{1}{K^\beta(\ln|\ln \epsilon|)^\beta},
	\end{align}
	where 
	$$
		\epsilon^2 
		=2\sup_{ k\in I, \, d\in\mathbb S}\Big(  k^2\|u^s_1(x,  k,  d) - u^s_2(   x,  k,    d)\|_{L^2(\Gamma_R)}^2
		+\|T (u^s_1(   x,  k,    d) - u^s_2(   x,  k,    d))\|_{L^2(\Gamma_R)}^2\Big)
	$$
	and $I = [K_0, K]$ with $K>M$,
	$\alpha = \frac{3}{2(3+2s)}$ and $\beta = \frac{s}{2(3 + 2s)}$.
\end{theorem}

The stability estimate \eqref{stability_1} implies the uniqueness. 
It consists of two parts: the data discrepancy and the high-frequency tail. The former is of the Lipschitz type. The latter decreases as $K$ increases which makes the problem have an almost Lipschitz stability. Overall, the result reveals that the problem becomes more stable when higher-frequency data is used.

\section{Inverse problem III: simultaneous recovery of source and potential}\label{sec: IP3}

This section is devoted to the co-inversion model of simultaneously reconstructing the source and potential from active measurements.
In this model, given the incident wave $u^{\rm inc}(\tilde{x}, x_3) = \mathrm{e}^{{\rm i}\sqrt{k^2 - \alpha_1^2}\tilde{x}\cdot d}\sin\alpha_1 x_3$ with incident direction $d\in\mathbb S$, the total field $u = u^{\rm inc}  + u^s$ satisfies the wave equation \cref{main_eq} and the boundary conditions \eqref{bc}.
We also assume that $f$ is real-valued and $V\geq 0$ in this section.

We are interested in the inverse problem of determining both $f$ and $V$ from $u|_{\Gamma_R}$.
To that end, we apply the method developed in \cite{WZ} to the current geometry of a waveguide. We first present a preliminary result for the subsequent analysis.
\begin{lemma}\label{tail}
	Let $d\in\mathbb S$ and $g\in H^1(\widetilde{B}_R)$ with $\text{supp}g\subset \widetilde{B}_R$ and $\|g\|_{H^{1}(\widetilde{B}_R)}\leq Q$. Then the following estimate holds: 
	$$
		\left|\int_{\widetilde{B}_R} g(x) \mathrm{e}^{{\rm i}\sqrt{k^2 - \alpha_1^2} \tilde{x}\cdot d}{\rm d}\tilde{x}\right| \leq  \frac{C(Q)}{k},
	$$
	where $C(Q)$ is a generic constant depending on $Q$.
\end{lemma}

\begin{proof}
	Since $d=(d_1, d_2)\in\mathbb S$, without loss of generality we assume that the $|d_1|\geq \sqrt{2}/2$. Using integration by parts yields
	\begin{align*}
		\int_{\widetilde{B}_R} g(\tilde{x}) \mathrm{e}^{{\rm i}\sqrt{k^2 - \alpha_1^2} \tilde{x}\cdot d}\,{\rm d}\tilde{x} &=  \frac{1}{{\rm i}\sqrt{k^2 - \alpha_1^2} d_1} \int_{\widetilde{B}_R} \frac{\partial}{\partial x_1} \left( \mathrm{e}^{{\rm i}\sqrt{k^2 - \alpha_1^2} \tilde{x}\cdot d}\right) g(\tilde{x})\,{\rm d}\tilde{x}\\
		& = -\frac{1}{ {\rm i}\sqrt{k^2 - \alpha_1^2} d_1} \int_{\widetilde{B}_R} \mathrm{e}^{{\rm i}\sqrt{k^2 - \alpha_1^2} \tilde{x}\cdot d}  \frac{\partial g}{\partial x_1}  {\rm d} \tilde{x},
	\end{align*}
	which gives by H\"older's inequality 
	\[
	\left|\int_{\widetilde{B}_R} g(\tilde{x}) \mathrm{e}^{{\rm i}\sqrt{k^2 - \alpha_1^2} \tilde{x}\cdot d}\,{\rm d}\tilde{x}\right| \leq C \|g\|_{H^1(\widetilde{B}_R)}\frac{1}{ k}.
	\]
	The proof is completed.
\end{proof}

In the case of the equation \eqref{main_eq} with an inhomogeneous term $f$, by multiplying both sides of \eqref{eqn_1} by 
$\mathrm{e}^{{\rm i}\sqrt{k^2 - \alpha_1^2}\tilde{x}\cdot d_1}\sin\alpha_1 x_3$ with $d_1\in\mathbb S$ and integrating by parts over $C_R$, 
\eqref{eqn_2} now becomes
\begin{align}\label{eqn_4}
	&\int_{\widetilde{B}_R} V \mathrm{e}^{{\rm i}\sqrt{k^2 - \alpha_1^2}(d + d_1)}{\rm d}\tilde{x} \notag\\
	&= \int_{\partial\widetilde{B}_R} \left(\partial_{\nu_{\tilde{x}}}u_1(\tilde{x}, k) \mathrm{e}^{{\rm i}\sqrt{k^2 - \alpha_1^2}\tilde{x}\cdot d_1}
	- {\rm i}\sqrt{k^2 - \alpha_1^2} d_1\cdot \nu_{\tilde{x}} u_1(\tilde{x}, k) \mathrm{e}^{{\rm i}\sqrt{k^2 - \alpha_1^2}\tilde{x}\cdot d_1} \right){\rm d}\tilde{x}\notag\\
	&\quad- \int_{C_R} V u^s \mathrm{e}^{{\rm i}\sqrt{k^2 - \alpha_1^2}\tilde{x}\cdot d_1}\sin\alpha_1 x_3{\rm d}x + \int_{C_R} f \mathrm{e}^{{\rm i}\sqrt{k^2 - \alpha_1^2}\tilde{x}\cdot d_1}\sin\alpha_1 x_3{\rm d}x\notag\\
	&=  \int_{\partial\widetilde{B}_R} \left(\partial_{\nu_{\tilde{x}}}u_1(\tilde{x}, k) \mathrm{e}^{{\rm i}\sqrt{k^2 - \alpha_1^2}\tilde{x}\cdot d_1}
	- {\rm i}\sqrt{k^2 - \alpha_1^2} d_1\cdot \nu_{\tilde{x}} u_1(\tilde{x}, k) \mathrm{e}^{{\rm i}\sqrt{k^2 - \alpha_1^2}\tilde{x}\cdot d_1}\right) {\rm d}\tilde{x}\notag\\
	&\quad+ \int_{\widetilde{B}_R} f_1 \mathrm{e}^{{\rm i}\sqrt{k^2 - \alpha_1^2}\tilde{x}\cdot d_1} {\rm d}\tilde{x}+\mathcal{O}\left(\frac{1}{k}\right).
\end{align}
By applying Lemma \ref{tail} we have
\[
\int_{\widetilde{B}_R} f_1 \mathrm{e}^{{\rm i}\sqrt{k^2 - \alpha_1^2}\tilde{x}\cdot d_1} {\rm d}\tilde{x} = \mathcal{O}\Big(\frac{1}{k}\Big),
\]
and then \eqref{eqn_4} becomes
\begin{align}\label{eqn_5}
	\int_{\widetilde{B}_R} V \mathrm{e}^{{\rm i}\sqrt{k^2 - \alpha_1^2}(d + d_1)}{\rm d}\tilde{x} 
	&=  \int_{\partial\widetilde{B}_R} \partial_{\nu_{\tilde{x}}}u_1(\tilde{x}, k) \mathrm{e}^{{\rm i}\sqrt{k^2 - \alpha_1^2}\tilde{x}\cdot d_1}\notag\\
	&\quad- {\rm i}\sqrt{k^2 - \alpha_1^2} d_1\cdot \nu_{\tilde{x}} u_1(\tilde{x}, k) \mathrm{e}^{{\rm i}\sqrt{k^2 - \alpha_1^2}\tilde{x}\cdot d_1} {\rm d}\tilde{x}
	+\mathcal{O}\Big(\frac{1}{k}\Big).
\end{align}
As \eqref{eqn_5} is similar to \eqref{eqn_2}, repeating the previous arguments we have that the active boundary data $\{u(x, k, d): x\in\Gamma_R, k\in I, d\in\mathbb S\}$
corresponding to the incident wave and source uniquely determine $V$. Once $V$ is known, we can also recover $f$.
In summary, we have the following uniqueness result:
\begin{theorem}
	The active multi-wavenumber boundary measurements 
	\[
	\{u(x, k, d): x\in\Gamma_R,\ k\in I,\ d\in\mathbb S\}
	\] 
	corresponding to the incident wave $\mathrm{e}^{{\rm i}\sqrt{k^2 - \alpha_1^2}\tilde{x}\cdot d}\sin\alpha_1 x_3$ uniquely determine $V$ with unknown $f$.
	As a consequence of the recovery of $V$, the source $f$ can be uniquely determined as well by the passive boundary measurements
	\[
	\{u(x, k_{n, j}): x\in\Gamma_R,  |k_{n, j}|\leq M\} \cup \{u(x, k): x\in\Gamma_R, k\in I\}.
	\]
\end{theorem}

\section{Conclusion}\label{sec: conclusion}

This work presents a resonance-free region and resolvent estimates for the resolvent of the Schr\"odinger operator in a planar waveguide in three dimensions. 
As an application, theoretical uniqueness and stability for several inverse scattering problems are established. The analysis only requires the limited aperture Dirichlet data at multiple wavenumbers. Moreover, we develop an effective Fourier-based reconstruction method for the inverse source problem. 
We believe that the method can be applied to the biharmonic wave equation in a planar waveguide with Navier boundary conditions and the geometry of tubular waveguides. A more challenging analytic problem is the direct and inverse elastic scattering in a waveguide. In this case, we may not have a straightforward Fourier decomposition such as \eqref{eqn_n} due to the coupled pressure and shear waves. We hope to report the relevant progress elsewhere in the future.


\end{document}